\documentclass[11pt]{article}
\usepackage{mymacros}
\usepackage{comment}
\usepackage{marginnote}
\usepackage{mathtools}
\usepackage{color}
\usepackage{xcolor}

\usepackage{soul}

\newcommand{\Lv}{\mathrm{Lv}}
\newcommand{\hLv}{\mathrm{hLv}}
\newcommand{\ShG}{\mathrm{ShG}}
\newcommand{\GFF}{\mathrm{GFF}}
\newcommand{\BRW}{\mathrm{BRW}}

\newcommand{\embdbalBRW}{\hat A}

\newcommand{\anc}[2]{{#1}^{#2}}
\newcommand{\Gen}[1]{D_{#1}}
\newcommand{\rGen}[2]{\cD_{#1,#2}}

\newcommand{\incr}[2]{\Delta_{#1,#2}}
\newcommand{\intRay}[1]{R_{#1}}
\newcommand{\horo}{\mathfrak{h}}

\newcommand{\MA}{M^A}
\newcommand{\MS}{M^S}
\newcommand{\tMA}{\tilde M^A}
\newcommand{\hMA}{\hat M^A}

\newcommand{\Mplus}{M^{+}}
\newcommand{\Mminus}{M^{-}}
\newcommand{\regGMC}[3]{M_{#1,#2}^{#3}}
\newcommand{\embdbalBRWregGMC}[3]{M_{#1, #2}^{\embdbalBRW^{#3}}}

\newcommand{\balBRWregGMC}[3]{M_{#1, #2}^{A^{#3}}}

\newcommand{\PLv}[1]{\P_{#1}^{\Lv} }	

\newcommand{\PShG}[1]{\P_{#1}^\ShG}
\newcommand{\ctLv}[1]{\nu_{#1}^\Lv}
\newcommand{\ctLmLv}[2]{\nu_{#1, #2}^\Lv}

\newcommand{\regmLv}[3]{\nu_{#1, #2, #3}^\Lv}

\newcommand{\ctLmGFF}[2]{\nu_{#1, #2}^\GFF}

\newcommand{\regmGFF}[3]{\nu_{#1,#2, #3}^\GFF }

\newcommand{\regmBRW}[2]{\nu_{#1, #2}^\BRW}
\newcommand{\regmLBRW}[3]{\nu_{#1, #2, #3}^\BRW }

\newcommand{\regmbalBRW}[3]{\nu_{#1, #2, #3}^\BRW}
\newcommand{\regmhLv}[3]{\nu_{#1, #2, #3}^\hLv}

\newcommand{\muA}{\mu^A}
\newcommand{\ctmGFF}[1]{\nu_{#1}^{\GFF}}
\newcommand{\ctmLv}[1]{\nu_{#1}^\Lv}
\newcommand{\embdPLv}[1]{\hat \P_{#1}^\Lv}
\newcommand{\embdELv}[1]{\hat \E_{#1}^\Lv}

\newcommand{\lpnorm}[3]{ \lVert #1 \rVert_{L^{#3}( \Gen{#2} )}}
\newcommand{\rlpnorm}[4]{\| #1 \|_{L^{#4}(\rGen{#2}{#3} ) }}

\newcommand{\base}{p}

\newcommand{\Cstar}{C_*}

\newcommand{\bC}{L}
\newcommand{\deli}{\delta_1}
\newcommand{\delii}{\delta_2}

\newcommand{\Ccomp}{K} 
\newcommand{\Clb}{K_0}  
\newcommand{\clb}{c_0}  
\newcommand{\BigC}{C}
\newcommand{\caux}{c'}

\newcommand{\TT}{{\mathfrak{T}}}
\newcommand{\Ray}{{\mbox{\tt Ray}}}

\newcommand{\ldlv}{k}

\newcommand{\evnt}{\cA_{v,a,k}}
\newcommand{\subdif}{\theta}
\newcommand{\mean}{m}

\title{Decay of correlations for the massless hierarchical Liouville model in infinite volume}

\date{}

\author{Michael Hofstetter \footnote{Department of Mathematics,
Weizmann Institute.  E-mail: {\tt michael.hofstetter@weizmann.ac.il}}  \and Ofer Zeitouni  \footnote{Department of Mathematics,
Weizmann Institute.  E-mail: {\tt ofer.zeitouni@weizmann.ac.il}} }

\begin{document}

\maketitle

\begin{abstract}
Let $A=(A_v)_{v\in \cT}$ be the balanced Gaussian Branching Random Walk on a $d$-ary tree $\cT$ and let $\MA$ be the multiplicative chaos with parameter $\gamma \in (0, \sqrt{2\log d})$ constructed from $A$.
In this work we establish the precise first order asymptotics of 
negative exponential moment of  $\MA$,
i.e.\ we prove that for $t_k = \lambda \base^k$ with $\lambda>0$ and $\base$ an explicit constant depending only on $\gamma$ and $d$, we have as $k \to \infty$,
\begin{equation*}
-\frac{1}{d^k} \log \E[e^{-\lambda \base^k \MA } ] \to  h(\lambda),
\end{equation*}
where $h\colon (0,\infty)\to \R$ is a non-explicit positive continuous function.

This result allows us to study the law of $A$ tilted by $e^{-t_k \MA}$ for particular values of $\lambda$, with $k\to \infty$.
In this setting we prove that the normalized $L^1$ norm of $A$ in generation $k-a$
is bounded and converges to $0$ when first $k\to \infty$ and then $a\to 0$.

As an application we prove that in this setting,
under the tilt $e^{-t_k \MA}$ and with $k\to \infty$, the Branching Random Walk $A$ exhibits a weak decay of correlations,
which is not present in the non-tilted model.

Our methods also apply to the usual Branching Random Walk $(S_v)_{v\in \cT}$ and with $\MA$ replaced by $\frac{1}{2}(\Mplus + \Mminus )$, where $\Mplus$ and $\Mminus$ are the multiplicative chaoses with parameter $\gamma \in (0, \sqrt{2\log d})$ constructed from $S$ and $-S$.
In that case we prove that, as $k\to \infty$, 
\begin{equation*}
-\frac{1}{d^k} \log \E[e^{- \frac{\lambda \base^k}{2}( \Mplus +  \Mminus) }] \to  \tilde h(\lambda),
\end{equation*}
where $\tilde h\colon (0,\infty)\to \R$ is again a non-explicit positive continuous function.

Our models are motivated by Euclidean field theory and can be seen as hierarchical versions of the massless Liouville and the sinh-Gordon field theory in infinite volume.
From this perspective our analysis sheds new light on the existence and the decay or correlations in these models, which are among the major open questions in this area.

\end{abstract}




\section{Introduction}
\subsection{Model and main results}

In this work we study the balanced Gaussian Branching Random Walk $(A_v)_{v\in \cT}$ on the rooted
$d$-ary regular tree $\cT=(V_\cT,E_\cT)$,
in the regime that the full mass of its multiplicative chaos $\MA$ attains a small value.
Here $V_\cT$ and $E_\cT$ denote the set of vertices and set of edges of $\cT$, and when it is clear from the context, we write $v\in \cT$ and $e \in \cT$ for vertices $v\in V_{\cT}$ and edges $e\in E_{\cT}$.
Our convention is that each vertex, starting from the root, has $d$ descendants. Thus, the degree of the root is $d$ and the degree of all other vertices is $d+1$.
We are particularly interested in the law of $A$ tilted by $e^{-t \MA}$ with $t>0$, which is given by
\begin{equation}
\label{eq:pt-liouville-definition}
\PLv{t}(B) = \frac{\E[ \mathbf{1}_{B} e^{-t\MA} ]}{\E[ e^{-t\MA} ]}, \qquad B \in \sigma(A_v \colon v\in \cT)
\end{equation}
and study the law of  $A$ under $\PLv{t}$ for vertices $v$ in generation $k\in \N$ depending on $t$, in the limit $t\to \infty$.
In Section \ref{sec:qft-models} we explain how this model and our analysis are related to a hierarchical version the massless Liouville field theory in infinite volume,
 hence the notation in \eqref{eq:pt-liouville-definition}.

It is essential that this model is defined with the balanced Branching Random Walk $(A_v)_{v\in \cT}$ rather than the usual Branching Random Walk denoted by $(S_v)_{v\in \cT}$.
The main reason for that is that the probability that $M^S$ (constructed from the standard Branching Random Walk $S$) is small is not that small, simply by making the first few generations of the Branching Random Walk very negative (this is not possible for the balanced case), see the beginning of Section \ref{sec-additional}.
However, since both $A$ and $S$ are closely related,
we introduce both models simultaneously.
For vertices $v,w\in V_{\cT}$ we denote by $v\leftrightarrow w \subseteq E_\cT$ the unique shortest path between $v$ and $w$ in $\cT$
and by $d_\cT(v,w)\equiv |v-w|_{\cT} $ the graph distance in $\cT$, i.e.\
\begin{equation}
d_{\cT}(v,w)  =   |v\leftrightarrow w|,
\end{equation}
where $|B|$ denotes the cardinality of sets $B\subseteq E_\cT$ or $B\subseteq V_\cT$.
We denote the root of $\cT$ by $o \in V_\cT$.
When $w = o$ we also write $d_\cT(v,w) = |v-o|_{\cT} \equiv |v|_{\cT}$.
Moreover, we write $\Gen{k}$ for the set of vertices in generation $k \in \N_0$, i.e.\
\begin{equation}
\Gen{k}(\cT) \equiv \Gen{k} = \{ v\in V_\cT \colon |v|_{\cT} = k \}.
\end{equation}
Note that $\Gen{0} = \{ o\}$ and $ |\Gen{k} | = d^k$.
For $w \in \Gen{k}$ and $m\in \N$, $m\leq k$ we denote by $\anc{w}{m}\in \Gen{k-m}$ the unique ancestor $m$ generations above $w$.
We write $\cT_n^w \subseteq \cT$ for $n\in \N_0$ for the subtree of $\cT$ with $n$ generations rooted at $w$ with the convention $\cT_k \equiv \cT_k^o$,
and $\Gen{n}^w$ for the subset of $\Gen{k+n}$ whose common ancestor at generation $k$ is $w$,
i.e.\ for $w\in \Gen{k}$ we have
\begin{equation}
\Gen{n}^w = \{ v\in \Gen{k+n} \colon \anc{v}{n} = w \}.    
\end{equation} 

Let now $(X_e)_{e\in E_\cT}$ be a collection of independent standard Gaussian random variables. Then the Branching Random Walk is the collection of random variables $(S_v)_{v\in \cT}$ with
\begin{equation}
\label{eq:brw-sum-along-paths}
S_v = \sum_{e \in v\leftrightarrow o} X_e,
\end{equation}
where we use the convention $S_o=0$.
When we are only interested in random variables at generation $n\in \N$,
we use the notation $(S_v)_{v\in \Gen{n}}$. 
We think of the collection $(X_e)_{e\in E_{\cT}}$ being attached to the edges $E_{\cT}$,
so that for every $v\in V_\cT$ the random variable $S_v$ is obtained by summing the independent increments $X_e$ along the path connecting $o$ and $v$.
Using the independence of $(X_e)_{e\in E_\cT}$ and the underlying tree structure, we have for $v,w \in \Gen{n}$
\begin{align}
\label{eq:variance-BRW}
\cov(S_v, S_w) &= n - \frac{1}{2} d_\cT (v,w).
\end{align}

For $i\in \N$ and $v\in \Gen{i}$ let $E(v)\subseteq E_{\cT}$ be the set of edges that connect $v$ with generation $i+1$, i.e.\
\begin{equation}
E(v) = \{ e\in E_{\cT} \colon e=\{v,w\} \text{ with } \anc{w}{1} = v \}.
\end{equation}

Similarly to the usual Branching Random Walk defined in \eqref{eq:brw-sum-along-paths} we define the balanced Branching Random Walk $(A_v)_{v\in \cT}$ by
\begin{equation}
\label{eq:def-brw-balanced}
A_v = \sum_{e\in o\leftrightarrow v}  Y_e ,
\end{equation}
where $(Y_e)_{e\in E_\cT}$ is now a collection of centred Gaussian random variables with variance $1$ and covariance $\E[ Y_e Y_{e'} ]=-1/(d-1)$ for $e,e' \in E(v)$ with $e\neq e'$,
and $\E[Y_e Y_{e'}] = 0$ for $e\in E(v)$, $e'\in E(v')$ with $v\neq v'$.
Alternatively, the family $(Y_e)_{e\in E_\cT}$ can be constructed from independent identically distributed centred Gaussian variables of variance $d/(d-1)$ conditioned to satisfy
\begin{equation}
\label{eq:balanced-brw-sum-to-0}
\sum_{e\in E(v)} Y_e = 0, \qquad v\in \cT.
\end{equation}

As for the standard Branching Random Walk $(S_v)_{v\in \cT}$ one can explicitly calculate the covariance of $(A_v)_{v\in \Gen{n}}$ thanks to the underlying tree structure.
In this case we have for $v,w \in \Gen{n}$
\begin{align}
\label{eq:variance-bal-BRW}
\cov(A_v, A_w) &= n - \frac{1}{2} d_\cT (v,w) - \mathbf{1}_{v\neq w}.
\end{align}
Note that the presence of the extra indicator $\mathbf{1}_{v\neq w}$ reflects the (negative) correlation at the descendents of the
branch-point of $v,w$.


For $n \in \N$ and $\gamma \geq 0$, we set
\begin{equation}
\label{eq:gmc-brw-volume-1-approximation}
M^{S,n} = \frac{1}{d^n} \sum_{v\in \Gen{n}} e^{\gamma S_v - \frac{\gamma^2}{2} n }
\end{equation}
and similarly
\begin{align}
\label{eq:gmc-balanced-brw-volume-1-approximation}
M^{A,n} &= \frac{1}{d^n} \sum_{v\in \Gen{n}} e^{\gamma A_v - \frac{\gamma^2}{2} n },
\end{align}
where the corrections in the exponent of \eqref{eq:gmc-brw-volume-1-approximation} and  \eqref{eq:gmc-balanced-brw-volume-1-approximation} come from \eqref{eq:variance-BRW} and \eqref{eq:variance-bal-BRW}.
Since $(M^{S,n})_{n\in \N}$ and $(M^{A,n})_{n\in \N}$ are non-negative martingales,
there exist random variables $M^S$ and $M^A$ such that as $n\to \infty$ 
\begin{equation}
\label{eq:bal-brw-and-brw-gmc-as-limit}
M^{A,n} \to \MA, \qquad \text{and} \qquad M^{S,n} \to \MS ,
\end{equation}
and it is known that for $\gamma \in ( 0, \sqrt{2\log d} )$ these limits are non-trivial,
a fact that goes back at least to \cite{MR433619}; see also \cite[Theorem 3.3]{MR3444654}.

With $\MA$ as in \eqref{eq:bal-brw-and-brw-gmc-as-limit},
we now present our main results for the measure $\PLv{t}$ in \eqref{eq:pt-liouville-definition} when $t\to \infty$.
Understanding this measure involves studying the large $t$ behaviour of the exponential moment of $-t\MA$,
for which we establish the precise first order asymptotics as $t\to \infty$.
As this result and its method of proof are of independent interest,
we state it separately in Theorem \ref{thm:liouville-mgf-asymptotics} below.

It is more convenient to express the asymptotics in terms of the logarithm of the exponential moment,
which we denote by
\begin{equation}
\label{eq:mgf-balanced-gmc}
\Lambda(t) = -\log \E[e^{-t\MA}], \qquad t\geq 0 .
\end{equation}
As we show in Lemma \ref{lem:liouville-mgf-asymptotics-bounds},
the leading order of $\Lambda$ as $t\to \infty$ is of order $O(t^\alpha)$,
where $\alpha$ depending on $\gamma$ is given by
\begin{equation}
\label{eq:exponents-first-order}
\kappa = \frac{2\log d}{\gamma^2} \qquad \text{ and } \qquad \alpha = \frac{\kappa}{\kappa +1} .
\end{equation}
Note that ${1}/{\sqrt{\kappa}}$ is the normalised parameter of the multiplicative chaos,
so that for $\gamma \in (0, \sqrt{2 \log d})$ we have $\kappa \in (1,\infty)$ and $\alpha \in (\frac{1}{2}, 1)$.

In fact, we are able to establish the leading order of $\Lambda$ at least along sequences $(t_k)_{k\in \N}$ of the form $t_k = \lambda \base^k$, where $\lambda>0$ and $\base$ is given by
\begin{equation}
\label{eq:base-sequence-tk}
\base=de^{\frac{\gamma^2}{2}} = e^{\frac{\gamma^2}{2 \log d}}.
\end{equation}
Note that with this notation we have $t_k^\alpha = \lambda^\alpha d^k$.

The statement for the first order asymptotics of $\Lambda$ then reads as follows.
\begin{theorem}
\label{thm:liouville-mgf-asymptotics}
There is a continuous function $h\colon (0,\infty) \to \R$ such that as $k\to \infty$ and with $\lambda>0$,
\begin{equation}
\label{eq:liouville-mgf-asymptotics}
-\frac{1}{d^k} \log \E[e^{-\lambda \base^k \MA} ] \to h(\lambda) ,
\end{equation}
where the convergence is uniform on compact sets.
\end{theorem}
Note that the function $h$ satisfies $h(\base \lambda)=d \cdot h(\lambda)$, 
and therefore $h$ is determined by its values for $\lambda\in [1,\base)$. 
In particular, with $\{x\} \coloneq x-\lfloor x\rfloor$ denoting the fractional part of $x\in \mathbb{R}$ and using 
the notation $\{t\}_\base = \base^{ \{\log_{\base}(t)\}}
\in [1,\base]$,
the convergence in \eqref{eq:liouville-mgf-asymptotics} gives that
\begin{equation}
\label{eq:t-lim-lambda}
\lim_{t\to \infty}
\Big| \frac{1}{t^\alpha} \Lambda(t) -\frac{ h(\{t\}_\base)}{(\{t\}_\base)^\alpha}\Big|=0.
\end{equation}
Indeed, for $t>0$, we have $t = \{ t \}_{\base} \base^k$ for some $k\in \N$ depending $t$.
Using $\base^\alpha = d$,
it follows that $t^\alpha = \{t\}_{\base}^\alpha d^k$.
Thus, we have
\begin{equation}
\frac{1}{t^\alpha} \Lambda(t) -\frac{h( \{t\}_\base)}{( \{t\}_\base)^\alpha}
=
\{t\}_\base^\alpha \Big(\frac{1}{d^k} \Lambda(\{t\}_\base d^k) - h( \{t\}_\base ) \Big).
\end{equation}
Since $\{t\}_\base$ is bounded and takes values in the compact set $[1,\base)$,
the convergence \eqref{eq:t-lim-lambda} follows from Theorem \ref{thm:liouville-mgf-asymptotics}.

We expect, without strong evidence, that
\begin{equation}
\label{eq:mgf-conjectured-t-asymptotics}
  h(\lambda) = \Cstar \lambda^\alpha \; \mbox{\rm for some constant 
  $\Cstar>0$},
\end{equation}
 which would imply the full convergence of $t^{-\alpha} \Lambda(t)$
to $\Cstar$. 
The restriction in Theorem \ref{thm:liouville-mgf-asymptotics}
to sequences $(t_k)_{k\in \N}$ of the form $t_k = \lambda \base^k$ comes from our proof method,
which is a large deviations argument involving the generations of the tree, and
therefore
can be traced back to the discreteness of the underlying $d$-ary tree $\cT$.
We note that in Section \ref{sec:qft-models} we explain how our analysis is related to a hierarchical version the massless Liouville field theory in infinite volume,
in which the sequence $(t_k)_{k\in \N}$ has an apparent counterpart.
Thus, from this point of view, it natural to consider the measures $(\PLv{t_k})_{k\in \N}$ with the particular choice $t_k= \lambda \base^k$, $\lambda>0$.

As an application of Theorem \ref{thm:liouville-mgf-asymptotics}
we establish a result on the size of values of the balanced Branching Random Walk $(A_v)_{v\in \cT}$ under $\PLv{t_k}$.
We quantify the size of the values by the $L^p$ norm defined as follows.
For a function $\varphi\colon \cT \to \R$ we define its $L^p$ norm at generation $n \in \N_0$ through 
\begin{equation}
\lpnorm{\varphi}{n}{p}^p = \frac{1}{|\Gen{n}|} \sum_{v\in \Gen{n}} |\varphi_v|^p.
\end{equation}
Our main result then states that, under $\PLv{t_k}$ with $t_k = \lambda \base^k$ and suitable choices of $\lambda>0$,
the $L^1$ norm of $(A_v)_{v\in\cT}$ at generation $k-a$ converges to $0$ when first $k\to \infty$ and then $a\to \infty$.
The ambiguity around the subsequence $(t_k)_{k\in \N}$ and the function $h$,
leads to a number of technical difficulties in the analysis of $\PLv{t_k}$.
To circumvent these problems, we prove in Lemma \ref{lem:f-variational-convex} that the function $f_\lambda\colon \R\to \R$, 
\begin{equation}
\label{eq:f-tau-definition-2}
f_\lambda(x) = h(\lambda e^{\gamma x} ) ,
\end{equation}
where $h$ is as in Theorem \ref{thm:liouville-mgf-asymptotics}, is convex. 
Then we are able to establish the claimed convergence of the $L^1$ norm for values of $\lambda>0$,
which have the property that $f_\lambda$ is strictly convex at $x=0$.
That such a $\lambda$ exists is due to the following argument.
First, the upper and lower sublinear bounds  on the growth of  $h$ at $+\infty$, which are due to Theorem \ref{thm:liouville-mgf-asymptotics}, see \eqref{eq:t-lim-lambda},
imply that for any $\lambda>0$, there is an $x_\lambda \in \R$ such that $f_\lambda$ is strictly convex at $x_\lambda$.
Next, the scaling relation $f_{\lambda e^y}(x)=f_\lambda(x+y)$ allows, at the cost of modifying $\lambda$,
to ensure that the strict convexity is at $0$.
Finally, the relation $f_{\base\lambda}(x)=d \cdot f_\lambda(x)$, which follows from Theorem \ref{thm:liouville-mgf-asymptotics},
shows that such $\lambda$ exists in the interval $[1,\base]$.
%
Thus, we have proved:
\begin{equation}
\label{eq:strictly-convex-at-0}
\mbox{\rm There exists $\lambda\in [1,\base]$ such that $f_\lambda(x)$ is strictly convex at $x=0$.}
\end{equation}
Then we can state our main result as follows.
\begin{theorem}
\label{thm:pt-liouville-decay-l1}
Fix $\lambda$ as in \eqref{eq:strictly-convex-at-0}. Let $\PLv{t_k}$ be as in \eqref{eq:pt-liouville-definition} with $t_k=\lambda \base^k$. 
Then there is a sequence $(\epsilon_a)_{a\in\N}$ with $\epsilon_a \to 0$ as $a\to \infty$ such that as $k\to \infty$
\begin{equation}
\label{eq:pt-liouville-decay-l1}
\PLv{t_k}( \lpnorm{A}{k-a}{1} \leq \epsilon_a ) \to 1.
\end{equation}
\end{theorem}

We emphasise that the convergence \eqref{eq:pt-liouville-decay-l1} is not true under the non-tilted law:
for any fixed $a\in\N$ and as $k\to\infty$, elementary arguments give
\begin{equation}
\label{eq:probability-L1-norm-bounded-limit}
\P( \lpnorm{A}{k-a}{1} \leq C ) \to 0
\end{equation}
for any constant $C>0$. Indeed, the (non-tilted) variance of $A_v$ for $v\in D_{k-a}$ is $k-a$, and thus $\P(|A_v|\leq C)\to_{k\to\infty} 0$.
It follows that $\E[ d^{-(k-a)} \sum_{v\in D_{k-a}} {\bf 1}_{|A_v|\leq C}]\to_{k\to\infty} 0$, implying \eqref{eq:probability-L1-norm-bounded-limit}.

\begin{remark}
\label{rem:conditional-liouville-decay-l1}
Instead of tilting with $e^{-t_k\MA}$ one may equivalently consider the law of the balanced Branching Random Walk conditional on the event $\MA \leq s$ and as $s \to 0$.
In this case the reference generation $\ell \in \N$ is related to $s = s_\ell$ by
\begin{equation}
\label{eq:ell-delta-relation}
s_\ell  = \tau e^{-\frac{\gamma^2}{2}\ell }
\end{equation}
for suitable $\tau>0$.
Then, using arguments similar to those used in the proof of Theorem
\ref{thm:pt-liouville-decay-l1},
one shows that there exists $\tau>0$ and a sequence $(\epsilon_a')_{a\in \N}$ with $\epsilon_a' \to 0$ as $a\to \infty$,
such that as first $\ell \to 0$ and then $a\to \infty$ we have
\begin{equation}
\P( \lpnorm{A}{\ell -a}{1} \leq \epsilon_a' \mid \MA \leq s_\ell ) \to 1.
\end{equation}
\end{remark}

As a consequence of Theorem \ref{thm:pt-liouville-decay-l1} 
we establish a weak decay of correlations for the balanced 
Branching Random Walk $(A_v)_{v\in \cT}$ under the measure $\PLv{t_k}$.
To state this result, it is more natural to adopt a different point of view 
and embed our trees in 
an infinite $d$-ary tree $\TT$ with one semi-infinite directed path, denoted $\Ray$, starting from a distinguished
vertex, called the root and denoted $\bar o$ (not to be confused with
the root $o$ of $\cT$), see Figure \ref{fig:brw-rooted-at-infty}.
This point of view can be traced back to \cite{MR1336708}, with the particular notation we use here following \cite{MR2365486}.
For vertices $v,w\in\TT$, we let $d_\TT(v,w)$ denote the length of the (unique) geodesic connecting $v$ and $w$.
We consider the geodesic as 
containing both $v$ and $w$, and thus its length equals the number of vertices in it minus one.
A vertex $w$ is an offspring of a vertex $v$ if $d_\TT(v,w)=1$ and either $d_\TT(w,\Ray)>d_\TT(v,\Ray)$ or $v,w\in \Ray$ and $d_\TT(v,\bar o)>d_\TT(w,\bar o)$.
In particular, the root is an offspring of its unique neighbour on $\Ray$.

For $v$ a vertex in $\TT$,
let $\intRay{v}\in \Ray$ denote the intersection of the geodesic connecting $v$ to $\Ray$ with $\Ray$, that is $d_\TT(v,\intRay{v})=d_\TT(v,\Ray)$.
For $v_1,v_2\in \TT$, let $\horo(v_1,v_2)$ denote the horocycle distance between $v_1$ and $v_2$ (possibly negative), defined
as the unique function $\horo(v_1,v_2)$ which equals 
$d_\TT(u,v_2)-d_\TT(u,v_1)$ for all vertices
$u$ such that both $v_1$ and $v_2$ are descendants of $u$.
Here, we call a vertex $w\in\TT$ a descendant
of $v$ if the geodesic connecting
$w$ to $v$ contains an offspring of $v$.
We also write $\horo(v)=\horo(\bar o,v)$.
The quantity $\horo(v)$, which may be
either positive or negative, is the {\it level} to which $v$ belongs;
see Figure \ref{fig:brw-rooted-at-infty}.
We denote by $\cL$ the set of vertices at level $0$.

We consider the sequence of subtrees $(\cR_k)_{k\in \N_0}$ such that $\cR_k$ is the unique $d$-ary subtree of $\TT$ rooted at $\bar o_k\in \Ray$ with $d_\TT(\bar o_k,\bar o)=k$.
Moreover, we denote for $k\in \N$ and $a \leq k$ by $\rGen{a}{k} = \Gen{k-a}(\cR_k)$ set of vertices at level $-a$ in $\TT$, which are descendants of $\bar o_k$, i.e.\
\begin{equation}
\label{eq:levels-of-infinite-tree}
\rGen{a}{k} = \{ v \in \cR_k \colon \horo(v) = -a \}.
\end{equation}
Note that those are at level $k-a$ for the subtree $\cR_k$ rooted at $\bar o_k$.

The vertices in $\rGen{0}{k} = \Gen{k}(\cR_k)$ form a subset of  $\cL$ which we denote by $\cL_k$.
Note that $\cL_k\subseteq \cL_{k+1}$, and that for any $v,w\in \cL$ there exists $k_0=k_0(v,w)$ so that $v,w\in \cR_k$ for all $k>k_0$.
For an illustration of the construction we refer to Figure \ref{fig:brw-rooted-at-infty}.
Every subtree $\cR_k$ gives rise to a balanced Branching Random Walk on $\cR_k\subseteq \TT$, which we denote by $(\embdbalBRW_w)_{w\in \cR_k}$.
We can realise all those balanced Branching Random Walks simultaneously, and when the underlying subtree it is clear from the context, we omit it from the notation.

For $\lambda\in [1,\base]$, set $t_k=\lambda \base^k$.
We write $\PLv{t_k}$ for the probability measure in \eqref{eq:pt-liouville-definition} with $\cT=\cR_k$.
This naturally induces a probability measure on $\sigma( \embdbalBRW_v \colon v\in \cL_k )$, which we denote by $\embdPLv{t_k}$ with corresponding expectation $\embdELv{t_k}$. 
We can now state our main result concerning decay of correlation.

\begin{figure}
\begin{center}
\def\svgwidth{0.5\columnwidth}
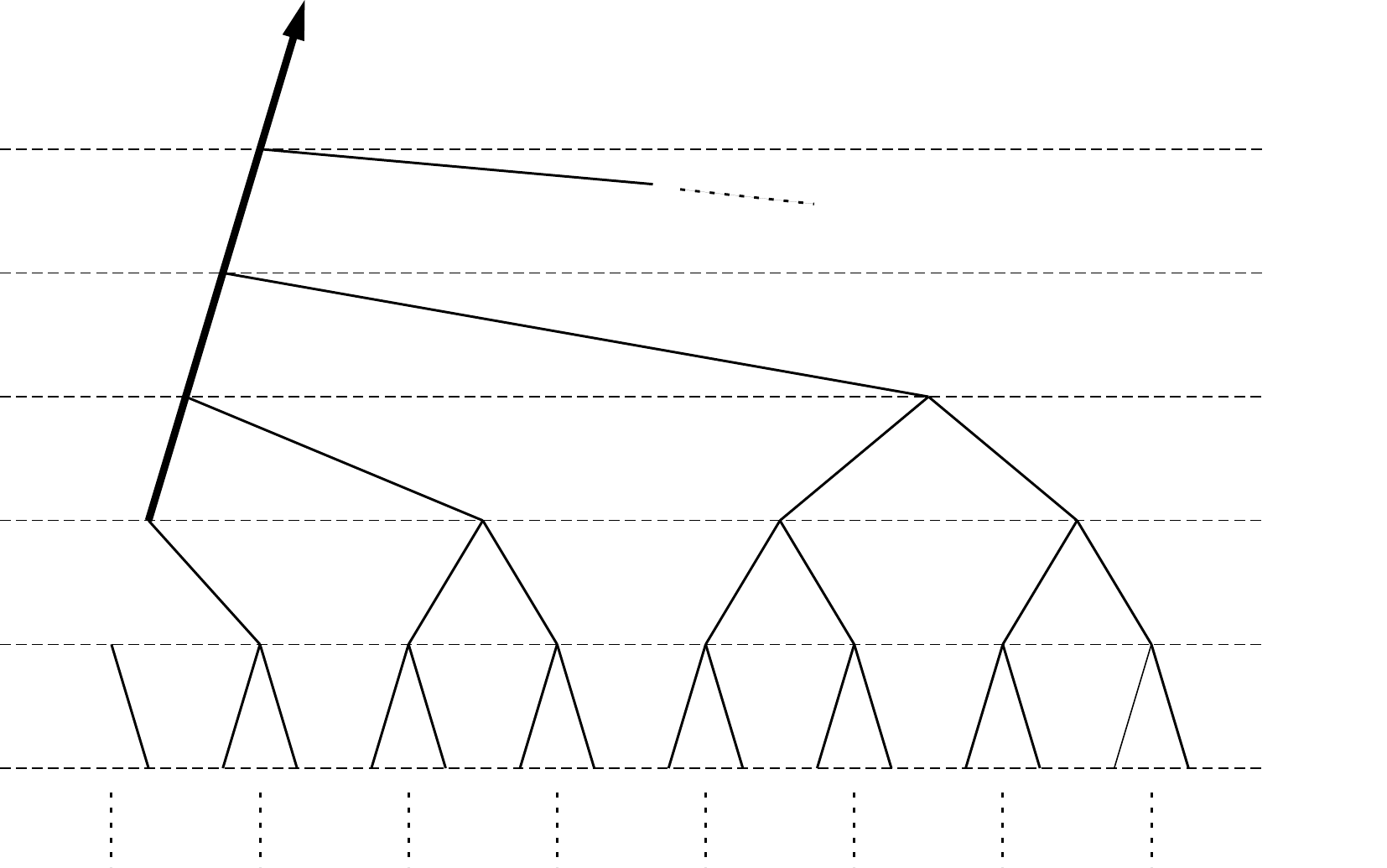
\end{center}
\caption{The tree $\TT$ in case $d=2$. The subtree rooted at $\bar o_1=R_v$ is $\cR_1$, and the subtree rooted at $\bar o_2$ is $\cR_2$. Note that $\cR_1$ is a subset of $\cR_2$.}
\label{fig:brw-rooted-at-infty}
\end{figure}

\begin{corollary}
\label{cor:decay-correlation-tree}
Fix $\lambda$ as in \eqref{eq:strictly-convex-at-0} and set  $t_k= \lambda \base^k$.
There exists a family of events
$(\evnt)_{v\in \cL, \, a, k\in \N }$ with $\embdPLv{t_k}(\evnt) \to 1$ as first $k\to \infty$ and then 
$a\to \infty$ such that
\begin{equation}
\label{eq:decay-correlation-tree}
\lim_{a \to \infty }\limsup_{k\to \infty} \sup_{w\in \cL\colon d_\TT(v,w)=2(a+1)} 
\big |\embdELv{t_k}[ \embdbalBRW_v \embdbalBRW_w  \mathbf{1}_{\evnt} ]  \big| = 0.
\end{equation}
\end{corollary}

It is important to note that the decay of correlations only holds on a 
sequence of events $\evnt$ with high probability.
This can traced back to the decay of the $L^1$ norm in Theorem \ref{thm:pt-liouville-decay-l1}, which does not give true correlations under $\PLv{t_k}$. 
For a decay of true correlations, it would suffice to prove Theorem \ref{thm:pt-liouville-decay-l1} for the $L^2$ norms $\lpnorm{A}{k-a}{2}$,
but our methods do not allow us to conclude that.


\subsection{Additional results}
\label{sec-additional}

Our methods also apply to some extent to the  Branching Random Walk $(S_v)_{v\in \cT}$.  
However, Theorem \ref{thm:liouville-mgf-asymptotics} is not true with $\MS$ in place of $\MA$.
Indeed, one can prove that 
\begin{equation}
\label{eq:gmc-unbalanced-decomposition}
\MS \stackrel{d}{=} e^{\gamma X - \frac{\gamma^2}{2}\frac{1}{d-1}} \MA, 
\end{equation}
where $X$ is an independent standard Gaussian random variable,
and thus, by Theorem \ref{thm:liouville-mgf-asymptotics} it is easy to see that there are constants $0<c\leq C$ such that
\begin{equation}
c(\log t)^2 \leq -\log \E[e^{-t\MS}] \leq C (\log t )^2.
\end{equation}
One can mimic the asymptotics \eqref{eq:liouville-mgf-asymptotics} for $(S_v)_{v\in \cT}$, if one replaces $\MA$ by $\frac{1}{2}(\Mplus + \Mminus )$.
The result, which is completely analogous to Theorem \ref{thm:liouville-mgf-asymptotics}, is as follows.

\begin{theorem}
\label{thm:sinh-mgf-asymptotics}
Let $(S_v)_{v\in \cT}$ be the Branching Random Walk on $\cT$ and let $\Mplus$, $\Mminus$ be the Gaussian multiplicative chaoses with respect to $S$ and $-S$.
Let $\lambda >0$ and let $\base$ be as in \eqref{eq:base-sequence-tk}.
Then there is a function $\tilde h\colon (0,\infty)\to \R$,
such that as $k\to \infty$
\begin{equation}
\label{eq:sinh-mgf-asymptotics}
-\frac{1}{d^k} \log \E[ e^{-\frac{\lambda}{2} \base^k (\Mplus + \Mminus)} ] \to \tilde h(\lambda) ,
\end{equation}
where the convergence is uniform on compact sets.
\end{theorem}

We emphasise that, apart from possibly different functions $h$ and $\tilde h$,
the asymptotics of the negative exponential moments of $\MA$ and $\frac{1}{2}(\Mplus+ \Mminus)$ are identical with the same values for $\alpha$ and $\kappa$,
which only depend $\gamma$ according to \eqref{eq:exponents-first-order}.

A phenomenon similar to Theorem \ref{thm:pt-liouville-decay-l1} is believed to occur also for the probability measure $\PShG{t}$ defined by
\begin{equation}
\label{eq:pt-sinh-definition}
\PShG{t} (B) = \frac{\E[\mathbf{1}_B e^{-t(\Mplus + \Mminus)}]}{\E [e^{-t (\Mplus + \Mminus)} ]}, \qquad  B \in \sigma( X_e \colon e\in \cT ) 
\end{equation}
but our methods are not strong enough to provide such a result.
More precisely, we expect that there is a sequence $(\tilde \epsilon_a)_{a\in \N}$ independent of $k$ and with $\tilde \epsilon_a \to 0$ as $a \to \infty$, such that as $k\to \infty$ 
\begin{equation}
\label{eq:pt-sinh-decay-l2}
\PShG{t_k} ( \lpnorm{S}{k -a}{2} \leq \tilde \epsilon_a ) \to 1.
\end{equation}

To prove \eqref{eq:pt-sinh-decay-l2} it seems that the key is to understand how an imbalance between $\Mplus$ and $\Mminus$ changes the second order of the large $t$ asymptotics of the exponential moment.
More precisely, a similar attempt as for Theorem \ref{thm:pt-liouville-decay-l1} leads to studying the asymptotics of 
\begin{equation}
\label{eq:mgf-sinh-u-definition}
\tilde \Lambda^u(t_k) = -\log \E [e^{- \frac{t_k}{2} (e^u \Mplus + e^{-u} \Mminus)} ] , \qquad u \in \R
\end{equation}
as $k\to \infty$.
For every $k\in \N$ this is a symmetric function in $u$, and one can even prove that it is convex in $u$. In particular, it has a minimum at $u=0$, so that
\begin{equation}
\label{eq:lower-bound-sinh-imbalance}
\tilde \Lambda^u(t_k) \geq \tilde \Lambda^0(t_k).
\end{equation}
The proof of Theorem \ref{thm:pt-liouville-decay-l1} is based
on the strict convexity, for appropriate $\lambda$,
of $\lim_k d^{-k}\Lambda(e^u t_k)$ in $u$,
which fails with $\tilde \Lambda^u(\cdot)$ replacing $\Lambda(e^u\cdot)$.
In fact, it is even true that
\begin{equation}
-\frac{1}{d^k} \log \E [e^{- \frac{t_k}{2} (e^u \Mplus + e^{-u} \Mminus)} ] \to \tilde h(\lambda) ,
\label{eq:sinh-imbalance}
\end{equation}
where $\tilde h$ is as in Theorem \ref{thm:sinh-mgf-asymptotics}.
Ultimately, $u$ is related to $(S_v)_{v\in \Gen{k-a}}$ and thus, 
since the lower bound \eqref{eq:lower-bound-sinh-imbalance} is independent of $u$,
it does not by itself lead to \eqref{eq:pt-sinh-decay-l2}.
In spite of \eqref{eq:sinh-imbalance}, the dependence on $u$ may enter in the second order of the asymptotics of $\tilde \Lambda^u$,
but our methods are not strong enough to yield this level of precision.

\subsection{Multiplicative chaos for Branching Random Walks}

The self-similarity of $\cT$ carries over to the Branching Random Walks $S$ and $A$ as follows.
For $k\in \N$, $w\in \Gen{k}$ and $n\in \N \cup \{ \infty \}$ we consider the Branching Random Walks $(S_v)_{v\in \cT_{k+n}}$ and $(A_v)_{v\in \cT_{k+n}}$ 
and write $(S_v^w)_{v\in \cT_n^w}$ and $(A_v^w)_{v\in \cT_n^w}$ for the Branching Random Walks on $\cT_n^w$ started from $w$,
i.e.\ for all $v\in \cT_n^w \subseteq \cT_{k+n}$ we have
\begin{equation}
S_v^w = S_v - S_w, \qquad A_v^{w} = A_v - A_w.
\end{equation}
Then $(S_v)_{v\in \cT_k}$ is independent of $(S_v^w)_{v\in \cT_n^w}$, and moreover
\begin{equation}
\label{eq:brw-subtree-same-law}
(S_v^w)_{v\in \cT_n^w} \stackrel{d}{=} (S_v)_{v\in \cT_n}.
\end{equation}
Similarly, we have that $(A_v)_{v\in \cT_k}$ is independent of $(A_v^w)_{v\in \cT_n^w}$ and
\begin{equation}
\label{eq:balanced-brw-subtree-same-law}
(A_v^w)_{v\in \cT_n^w} \stackrel{d}{=} (A_v)_{v\in \cT_n}.
\end{equation}

The identities \eqref{eq:brw-subtree-same-law} and \eqref{eq:balanced-brw-subtree-same-law} give rise to the following recursions for $\MS$ and $\MA$,
which are useful in the sequel.
Let $(\MS_v)_{v\in \Gen{k}}$ and $(\MA_v)_{v\in \Gen{k}}$ be collections of independent and identically distributed random variables with $\MA_v\stackrel{d}{=} \MA$ and independent of $(A_v)_{v\in \Gen{k}}$,
and analogously $\MS_v\stackrel{d}{=} \MS$ and independent of $(S_v)_{v\in \Gen{k}}$. Then we have
\begin{equation}
\label{eq:brw-gmc-recursion}
\MS \stackrel{d}{=} \frac{1}{d^k} \sum_{v\in \Gen{k}} e^{ \gamma S_v - \frac{\gamma^2}{2} k } \MS_v
\end{equation}
and similarly
\begin{equation}
\label{eq:balanced-brw-gmc-recursion}
\MA \stackrel{d}{=} \frac{1}{d^k} \sum_{v\in \Gen{k} } e^{ \gamma A_v - \frac{\gamma^2}{2} k } \MA_v ,
\end{equation}
where now $A$ is the Branching Random Walk in \eqref{eq:def-brw-balanced}.
Indeed, to prove \eqref{eq:brw-gmc-recursion} we let $k, n \in \N$ and use \eqref{eq:brw-subtree-same-law} to obtain
\begin{align}
\label{eq:brw-gmc-recursion-prelimit}
M^{S, k+n} &= \frac{1}{d^{k+n}} \sum_{ w\in \Gen{k+n} } e^{\gamma \big( S_{\anc{w}{n}} + (S_w - S_{\anc{w}{n}} )  \big)  - \frac{\gamma^2}{2} (n+k) } \nnb
&\stackrel{d}{=} \frac{1}{d^k} \sum_{v \in \Gen{k}} e^{\gamma S_v - \frac{\gamma^2}{2} k } \frac{1}{d^n} \sum_{w\in \Gen{n}^v} e^{\gamma S_w^v - \frac{\gamma^2}{n}} 
= \frac{1}{d^k} \sum_{ v\in \Gen{k} } e^{\gamma S_v - \frac{\gamma^2}{2} k } M_v^{S,n},
\end{align}
where $(M_v^{S,n})_{v\in \Gen{k}}$ are independent and distributed as $M^{S,n}$ in \eqref{eq:gmc-brw-volume-1-approximation}.
Thus, since  $(S_v)_{v\in \Gen{k}}$ is independent of $(M_v^{S,n})_{v\in \Gen{k}}$,
the claimed identity \eqref{eq:brw-gmc-recursion} follows from \eqref{eq:brw-gmc-recursion-prelimit} by taking the limit $n\to \infty$.
The proof of \eqref{eq:balanced-brw-gmc-recursion} is identical with $A$ in place of $S$.

\subsection{Notation}

We use the the standard Landau big-$O$ and little-$o$ notation,
i.e.\ for real valued functions $f,g\colon [0,\infty) \to \R$ we write $f= O(g)$ if $|f(u)/g(u)| \leq C$ for all $u\geq 0$ for some universal constant $C>0$,
and $f = o_u(g)$ if $|f(u)/g(u)| \to 0$ for $u\to \infty$.
When it is clear from the context with respect to which parameter the little-$o$ notation is to be understood, we omit the subscript and write $f = o(g)$.  
Moreover, we use the standard notation $\R_0^+ = [0,\infty)$.

\section{Relation to infinite volume models in Euclidean field theory}
\label{sec:qft-models}

In this section we explain how the probability measures $(\PLv{t_k})_{k\in \N}$ from \eqref{eq:pt-liouville-definition} are related to a massless Liouville Euclidean field theory in the plane,
which is a probability measure on the space of dis\-tri\-butions $S'(\R^2)$ formally defined as the weak limit $m\to 0$ and then $L\to \infty$ of the probability measures
\begin{equation}
\label{eq:liouville-qft-formal-definition}
\ctLmLv{L}{m} (d\phi) \propto \exp \Big[ -\lambda \int_{\Lambda_{L}} e^{ \sqrt{\beta} \phi(x) } dx  \Big] \ctLmGFF{L}{m}(d\phi),
\end{equation}
where $\Lambda_L = [0,L)^2 \subseteq \R^2$ is a box of side length $L$,\footnote{
It is also possible to consider other choices for the box $\Lambda_L$,
e.g.\ one which exhausts the full plane as $L\to \infty$.
This however leads to a more elaborate notation.} $\lambda>0$ and $\beta \in (0, 8\pi)$ are coupling constants,
and $\ctLmGFF{L}{m}$ is the law of the massive Gaussian Free Field with average $0$ on $\Lambda_L$ to be defined below.
Our analysis suggests, but does not prove, that the sequence $(\ctLmLv{L}{m})_{L>0, m>0}$ is tight.
This would give rise to a massless Liouville field theory in infinite volume,
which has not been rigorously constructed so far; we emphasize however that the removal of the zero mode is not innocent. We refer further to the discussion in Section \ref{sec:literature}.
The most interesting case is when the reference measure $\ctLmGFF{L}{m}$ is log-correlated,
for which we restrict attention to dimension $2$.
Since we eventually want to draw a comparison between $\ctLv{L}$ and $\PLv{t_k}$,
we choose $L=2^k$, $k\in \N$ and assume that $\PLv{t_k}$ is constructed from a $4$-ary balanced Branching Random Walk.

The expression in \eqref{eq:liouville-qft-formal-definition} is ill-defined, since the Gaussian Free Field is a.s.\ not a regular function on $\R^2$,
but takes values in the space of distributions $S'(\R^2)$.
To give meaning to \eqref{eq:liouville-qft-formal-definition} a regularisation procedure is necessary,
which then allows to construct $\ctLmLv{L}{m}$ as a weak limit.

To make the connections to the balanced Branching Random Walk $(A_v)_{v\in \cT}$ and the measure $\PLv{t_k}$ evident,
we choose the following space regularisation:
let $\Lambda_{L, \epsilon} = \Lambda_L \cap \epsilon \Z^2$ be the discretisation of $\Lambda_L$ of mesh size $\epsilon>0$,
where we assume throughout this section that $\epsilon = 2^{-n}$ for $n\in \N$.
Moreover, we write $X_{\epsilon} = \{ \varphi \colon \Lambda_{\infty,\epsilon}  \to \R\}$ for the set of real valued functions on $\Lambda_{\infty,\epsilon} \equiv \R^2 \cap (\epsilon\Z)^2$.
Then the discretisation of the measure in \eqref{eq:liouville-qft-formal-definition} is the probability measure on $X_{\epsilon}$ with density
\begin{equation}
\label{eq:discrete-massive-liouville-density}
\regmLv{L}{m}{\epsilon}(d\phi)
\propto
\exp{ \Big[-\lambda  \,\epsilon^2 \sum_{x \in \Lambda_{L,\epsilon} } \wick{e^{\sqrt{\beta} \phi_x}}_\epsilon\Big]} d\regmGFF{L}{m}{\epsilon} (\phi),
\end{equation}
where now $\regmGFF{L}{m}{\epsilon}$ is the density of the massive discrete Gaussian Free Field with average zero on $\Lambda_{L,\epsilon}$
and the regularisation of the exponential function is given by
\begin{equation}
\label{eq:exp-wick-ordering}
\wick{ e^{\sqrt{\beta} \phi_x} }_\epsilon = \epsilon^{\beta/4\pi} e^{\sqrt{\beta}\phi_x}.
\end{equation}

The reference measure $\regmGFF{L}{m}{\epsilon}$ in \eqref{eq:discrete-massive-liouville-density} is as follows.
Let $\nu_{m,\epsilon}^{\GFF}$ be the discrete Gaussian Free Field on $\R^2$ with mass $m>0$,
i.e.\ the centred Gaussian field on $X_{\epsilon}$ with covariance
\begin{equation}
c_{m,\epsilon} =
(-\Delta^\epsilon + m^2 )^{-1},
\end{equation}
where $\Delta^\epsilon$ is the lattice Laplacian on $\R^2 \cap \epsilon \Z^2$ given by $\Delta^\epsilon f (x) = \frac{1}{\epsilon^2}\sum_{y\sim x}  ( f(y) - f(x) )$.
Note that as $\epsilon \to 0$, we have (by the random walk representation of the covariance as a Green function)
\begin{align}
\label{eq:variance-massive-gff}
&c_{m,\epsilon}(x,y)=\delta_{x,y}+\sum_{t=1}^\infty \frac{1}{2\pi t} e^{-|x-y|^2/2t} e^{-mt} \sim 
\frac{1}{2\pi} \log \frac{1}{m (|x-y|\vee \epsilon)}+O_{\epsilon,m}(1),\nnb
& c_{m,\epsilon} (x,x) \sim \frac{1}{2\pi} \log \frac{1}{m\epsilon}+O_{\epsilon,m}(1).
\end{align}
For $f \in X_{\epsilon}$ let $\avg{f}_{\Lambda_{L,\epsilon}} = |\Lambda_{L,\epsilon}|^{-1} \sum_{x\in \Lambda_{L,\epsilon}} f(x)$ be the average of $f$ on $\Lambda_{L,\epsilon}$.
Let now $\Phi^{m} \sim \nu_{m,\epsilon}^{\GFF}$ and define 
\begin{equation}
\label{eq:massive-gff-without-avarage-on-L-box}
\Phi^{L,m} = \Phi^{m} - \avg{\Phi^{m}}_{\Lambda_{L,\epsilon}}.
\end{equation}
Then the reference measure $\regmGFF{L}{m}{\epsilon}$ in \eqref{eq:discrete-massive-liouville-density} is given by the law of $\Phi^{L,m}$.
Note that $\Phi^{L,m}$ is is a random field on $\Lambda_{\infty, \epsilon}$ and as $\epsilon\to 0$ converges to a Gaussian field on $\R^2$.
We mostly are interested in the restriction of this field to $\Lambda_L$.
We often omit the regularisation parameter $\epsilon$ in the notation, when it is irrelevant to the argument and doing so  does not cause confusion. Using \eqref{eq:variance-massive-gff} we then obtain that
 the covariance of $\Phi^{L,m}$  on $\Lambda_{L,\epsilon}$ is of the form
\begin{align}
\label{eq:variance-L-gff}
\cov(\Phi_x^{L,m}, \Phi_y^{L,m}) &= \frac{1}{2\pi} \log\frac{L}{|x-y|\vee \epsilon} + O_{L,\epsilon}(1), \qquad x,y \in \Lambda_{L,\epsilon}, 
\end{align}
and that $(\Phi^{L,m})_{m>0, \epsilon>0}$, restricted to $\Lambda_{L,\epsilon}$,  converges in distribution to a log-correlated Gaussian field $\Phi^{L,0}$ on $\Lambda_L$ as $\epsilon \to 0$ and then $m\to 0$ with
\begin{align}
\cov(\Phi_x^{L,0}, \Phi_y^{L,0}) &= \frac{1}{2\pi} \log \frac{L}{|x-y|} + O(1), \qquad x,y\in \Lambda_L, x\neq y.
\end{align}
The regularisation procedure \eqref{eq:exp-wick-ordering} is known as Wick ordering at the fixed length scale $L_0=1$.
%
%
With $\phi \sim \regmGFF{L}{m}{\epsilon}$ and \eqref{eq:variance-L-gff} it can be rewritten as
\begin{equation}
\label{eq:wick-ordering-gmc}
\wick{ e^{\sqrt{\beta} \phi_x} }_\epsilon
= e^{\sqrt{\beta}\phi_x - \frac{\beta}{4\pi} \log \frac{1}{\epsilon}}
= e^{\sqrt{\beta}\phi_x - \frac{\beta}{4\pi} \log \frac{L}{\epsilon} + \frac{\beta}{4\pi} \log L }
= L^{\beta/4\pi} e^{\sqrt{\beta} \phi_x -\E[(\sqrt{\beta}\phi_x)^2] + O(1)}.
\end{equation}
Thus, in this particular case, the Wick ordering gives, in the limit $\epsilon\to 0$, the Gaussian multiplicative chaos on $\Lambda_L$ of the Gaussian Free Field with parameter $\beta$, and the fact that we renormalised at length scale $L_0=1$ leads to the volume dependent multiplicative correction $L^{\beta/4\pi}$.
%
Writing
\begin{equation}
\label{eq:gff-gmc-approximation-full-mass}
\regGMC{L}{\epsilon}{\sqrt{\beta}}(\phi)
= \epsilon^2 \sum_{x\in \Lambda_{L,\epsilon}} e^{\sqrt{\beta}\phi_x - \frac{\beta}{4\pi} \log \frac{1}{\epsilon}}
\end{equation}
for the full mass of the multiplicative chaos, 
we see from \eqref{eq:variance-L-gff} that the range $(0,8\pi)$ for $\beta$ is precisely\footnote{For standard GFF with variance $-\log|t-s|+O(1)$, 
the standard range for non-degeneracy of the GMC is $(0,\sqrt{2d})$ when the ambient space is $\R^d$, see
\cite{MR3274356}. Hence, with our normalization
of the GFF, the correct range is $(0,\sqrt{2\pi\cdot 2d})=(0,\sqrt{8\pi})$ since $d=2$.}   the interval for which $\regGMC{L}{\epsilon}{\sqrt{\beta}}(dx)$ has a non-trivial limit as $\epsilon\to 0$.

Now, the density \eqref{eq:discrete-massive-liouville-density} becomes
\begin{equation}
\label{eq:discrete-gff-liouville-density-with-gmc}
\regmLv{L}{m}{\epsilon}(d\phi)
\propto
\exp{ \big[ -\lambda \regGMC{L}{\epsilon}{\sqrt{\beta}} (\phi) \big]  } d\regmGFF{L}{m}{\epsilon} (\phi).
\end{equation}
We emphasize that the substraction of the zero mode, see \eqref{eq:massive-gff-without-avarage-on-L-box}, changes dramatically the model. On the Gaussian free field side, it destroys the conformal invariance. On the asymptotics side, as explained in the beginning of Section \ref{sec-additional}, the tails of the law of the total mass of the Gaussian multiplicative chaos (GMC) are very different for the model with and without zero mode.

We next define a measure analogously to \eqref{eq:discrete-gff-liouville-density-with-gmc}, but now with respect to a reference measure $\regmbalBRW{L}{m}{\epsilon}$ that is constructed from a suitable Branching Random Walk.
This measure is then closely related to $\PLv{t_k}$ for $t_k= \lambda \base^k$.
To this end, we start with a massive Branching Random Walk\footnote{There are other ways to introduce the massive Branching Random Walk on $\Lambda_{\infty,\epsilon}$,
for example starting with embedding the tree $\TT$ (illustrated in Figure \ref{fig:brw-rooted-at-infty})
in $\R^2$, identifying $\Lambda_{\infty,\epsilon}$ with the vertices corresponding to $\horo=\log_2(\epsilon)$,
and using the killed Laplacian on that graph with killing parameter $m$ at levels with  $\horo \leq 0$.
It is not hard to see that the results are similar.}
on $\Lambda_{\infty, \epsilon}$ obtained as follows.
Let $m=2^{-K}$, where we assume throughout that $K>k$.
Then we subdivide $\R^2$ into boxes of side length $1/m=2^K$ with lower left corners in $2^{K}\Z^2$
and further discretise the subboxes by a lattice of width $\epsilon=2^{-n}$. 
Into each subbox we embed an independent standard Branching Random Walk of depth $K+n$ in the standard way respecting the dyadic subdivision (that is, each vertex of the tree at depth $K+\ell$ corresponds to a box of the form $[j_1 2^{-\ell},(j_1+1) 2^{-\ell})\times 
[j_2 2^{-\ell},(j_2+1) 2^{-\ell})$ for some $(j_1,j_2)$).
We call this infinite collection of Branching Random Walks
the massive Branching Random Walk with (idealized) mass $m$ and denote its law by $\regmBRW{m}{\epsilon}$.
Note that  the construction creates independent blocks of side $m^{-1}=2^{K}$, which respects the notion of a hierarchical model, whereas a truly massive model would have an exponential decay of correlations at scale $2^K$. 

Let  $W=[0,2^{K})^2$ be the subbox with lower left corner $(0,0)\in\R^2$
and let $W_\epsilon=W \cap (\epsilon\Z^2)$ be its discretisation. 
Note that since $K>k=\log_2 L$, we have $\Lambda_L \subseteq W$.
Let $\TT$ be the infinite tree introduced above Corollary \ref{cor:decay-correlation-tree} with a semi-directed $\Ray$ emanating from the root $\bar o$,
see Figure \ref{fig:brw-rooted-at-infty},
and let $(\cR_k)_k$ be the sequence of nested subtrees (which in particular all contain $\bar o$).
We embed into $(\R_0^+)^2$ the tree $\TT$ so that $\bar o$ is identified with the square $[0,1)^2 \in \R^2$,
see Figure \ref{fig:embedded-brw-decomposition}.
Note that the subtree $\cR_{K}$ is then embedded into $W$
and that as mentioned above, the embedding gives an identification between points $x \in W_\epsilon$ and vertices $v\in \Gen{K+n}(\cR_{K})$. Note also that our construction of the massive Branching Random Walk is equivalent to setting all edge variables for edges $(v,w)$ with $\mathfrak{h}(v),\mathfrak{h}(w)<-K$ to equal $0$.

Let now $\hat S^m \sim \regmBRW{m}{\epsilon}$.
To obtain the hierarchical analogue of $\regmGFF{L}{m}{\epsilon}$ in \eqref{eq:discrete-massive-liouville-density},
we subtract analogously to \eqref{eq:massive-gff-without-avarage-on-L-box} the average of $\hat S^m$ on $\Lambda_{L,\epsilon}$ and define
\begin{equation}
\label{eq:embedded-BRW-avg-box}
\hat S^{L,m}= \hat S^m - \avg{\hat S^m}_{\Lambda_{L,\epsilon}}. 
\end{equation}
Note that $\hat S^{L,m}$ is a Gaussian random field on $\Lambda_{\infty,\epsilon}$.
As $\epsilon \to 0$ and then $m\to 0$,
this field, restricted to $\Lambda_L$ converges to a  Gaussian field on $\R^2$ in the usual sense of distributions.
We denote the law of $\hat S^{L,m}$ by $\regmLBRW{L}{m}{\epsilon}$.


By construction, all points $x\in \Lambda_{L,\epsilon} \subseteq W_\epsilon$ have a common ancestor $w_0 \in \Gen{K-k}(\cR_K)$, $w_0\in \Ray$.
Let $S^{w_0,L,m}$ be the Branching Random Walk on the subtree 
$\cR_{k}\subseteq \cR_K$
starting from $w_0$, so that we have for $x\in \Lambda_{L,\epsilon}$
\begin{equation}
\label{eq:embedded-BRW-decomposition}
\hat S_x^{m} = \hat S_{w_0}^{m} +  S_x^{w_0, L,m}.
\end{equation}
Then we have
\begin{equation}
\label{eq:avg-embedded-BRW-on-box}
\avg{\hat S^{m}}_{\Lambda_{L,\epsilon}} = \frac{1}{|\Lambda_{L,\epsilon}|} \sum_{x\in \Lambda_{L,\epsilon}} \hat S_x^{m}
= \hat S_{w_0}^{m} + \frac{1}{|\Lambda_{L,\epsilon}|} \sum_{x\in \Lambda_{L,\epsilon}} S_x^{w_0, L,m}.
\end{equation}
In what follows we denote by $\embdbalBRW$ the restriction of $\hat S^{L,m}$ to $\Lambda_{L,\epsilon}$.
By \eqref{eq:embedded-BRW-avg-box}, and \eqref{eq:avg-embedded-BRW-on-box} we have for $x\in \Lambda_{L,\epsilon}$ (corresponding to vertices $v\in \Gen{k+n}(\cR_k)$)
\begin{equation}
\label{eq:bal-brw-for-hierarchical-liouville}
\embdbalBRW_x = \hat S_{w_0}^{m} + S_x^{w_0, L,m} - \hat S_{w_0}^m - \frac{1}{|\Lambda_{L,\epsilon}|} \sum_{x\in \Lambda_{L,\epsilon}}  S_x^{w_0,L,m}  
=  S_{x}^{w_0,L,m} - \frac{1}{\Lambda_{L,\epsilon}} \sum_{x\in \Lambda_{L,\epsilon}}  S_x^{w_0,L,m} .
\end{equation}
It follows  that for $m<2^{-L}$, the law $\regmLBRW{L}{m}{\epsilon}$ does not depend on $m$.

The expression on the right hand side of \eqref{eq:bal-brw-for-hierarchical-liouville} is closely related to a balanced Branching Random Walk of depth $k+n$ as defined in \eqref{eq:def-brw-balanced}.
To make this relation evident,
we consider a Branching Random Walk $(S_v)_{v\in \cR_k, \horo(v) \leq n}$ and construct a field $(\hat A_v^{k+n})_{v\in \cR_k,  \horo(v) \leq n}$ by replacing the value $S_v$ by its average over the descendants of $v$ in level $n$. 
More precisely, we define for each vertex $v\in \cR_{k}$ with $\horo(v) \leq n$
\begin{equation}
\hat A_v^{k+n} = d^{-(n-\horo(v) )} \sum_{w \in \Gen{n-\horo(v)}^v(\cR_k) } S_w .
\end{equation}
For an edge $e=(\anc{v}{1},v)$, $v\in \cR_k \setminus \{ \bar o_k\}$,
where we recall that $\anc{v}{1}$ denotes the ancestor of $v$ in the level above $v$,
we define
\begin{equation}
Y_e^{k+n}= \hat A_v^{k+n}- \hat A_{v^1}^{k+n}.
\end{equation}
Similarly to \eqref{eq:def-brw-balanced}, but now with $Y_e^{k+n}$ in place of $Y_e$, we obtain
\begin{equation}
\label{eq:def-bal-BRW-n}
\hat A_v^{k+n} =\sum_{e\in \bar o_k \leftrightarrow v}  Y_e^{k+n}.
\end{equation}
One can verify that $\hat A^{k+n}$ is a balanced Branching Random
Walk whose increments $(Y_e^{k+n})_{e\in \cR_k}$ have variances
\begin{equation}
\label{eq:variance-ynk}
\var(Y_e^{k+n}) = 1 - d^{-(n+k-i)}, \qquad e\in E(v) \text{ with } v \in \Gen{i}(\cR_k),
\end{equation}
and so that $(Y_e^{n+k})_{e\in E(v)}$ is an independent collection of centred Gaussian random variables of variance $(1 - d^{-(n+k-i)})/(1-1/d)$,
subject to
\begin{equation}
\label{eq:balanced-brw-n-sum-to-0}
\sum_{e \in {E(v)}} Y_e^{k+n} = 0.
\end{equation}
Moreover, for a fixed $k$ the law of $(\hat A_v^{k+n})_{v\in \cR_k \colon \horo(v) \leq 0}$ converges as $n\to\infty$ to the balanced Branching Random Walk $(\embdbalBRW_v)_{v\in \cR_k \colon \horo(v) \leq 0}$.
One then obtains that for $x\in \Lambda_{L,\epsilon}$,
\begin{equation}
\label{eq:embd-balanced-brw-variance}
\var(\hat S_x^{L,m}) = \var( \embdbalBRW_x  )
= k + n - \frac{1 - d^{-(k+n)}}{d-1} 
= \log_2\frac{L}{\epsilon} + O(1)
= \frac{1}{\log 2} \log \frac{L}{\epsilon} + O(1),
\end{equation}
and hence, as $\epsilon \to 0$, the variances of the Gaussian Free Field in \eqref{eq:variance-L-gff} and the balanced Branching Random Walk in \eqref{eq:embd-balanced-brw-variance} agree up to a multiplicative constant.
Analogously to \eqref{eq:gff-gmc-approximation-full-mass} the
multiplicative chaos of $S^{L,m}$ in $\Lambda_{L,\epsilon}$ is then given by
\begin{equation}
\label{eq:brw-liouville-gmc-approx}
\regGMC{L}{\epsilon}{\gamma}(\phi) = \epsilon^2 \sum_{x\in \Lambda_{L, \epsilon}} e^{\gamma \phi_x -\frac{\gamma^2}{2\log 2}  \log \frac{1}{\epsilon} + O(1)} , \qquad \phi \sim \regmLBRW{L}{m}{\epsilon},
\end{equation}
where
$\gamma \in (0, \sqrt{2 \log 4})$.
As for the Gaussian Free Field, we renormalise at the fixed length scale $L_0=1$,
for which the length parameter $L$ does not appear in the exponent of  \eqref{eq:brw-liouville-gmc-approx}.
We see that $\regGMC{L}{\epsilon}{\gamma}(\phi)$ only depends on $\phi_x$ for $x\in \Lambda_{L,\epsilon}$,
and thus,
\begin{equation}
\label{eq:BRW-GMC-depending-on-box-only}
\regGMC{L}{\epsilon}{\gamma}(\hat S^{L,m}) = \regGMC{L}{\epsilon}{\gamma}(\embdbalBRW).
\end{equation}

We can now define analogously to \eqref{eq:discrete-gff-liouville-density-with-gmc} a measure $\regmhLv{L}{m}{\epsilon}$ by
\begin{equation}
\label{eq:discrete-brw-liouville-with-gmc}
\regmhLv{L}{m}{\epsilon}(d\phi) \propto e^{-\lambda \regGMC{L}{\epsilon}{\gamma}(\phi)} \regmbalBRW{L}{m}{\epsilon}(d \phi) ,
\end{equation}
which we refer to as the discretised hierarchical Liouville model.
Identifying $\gamma$ in \eqref{eq:discrete-brw-liouville-with-gmc} with $\sqrt{ \beta }$ in \eqref{eq:discrete-gff-liouville-density-with-gmc} the measure $\regmhLv{L}{m}{\epsilon}$ is a hierarchical version of $\regmLv{L}{m}{\epsilon}$,
in the sense that the underlying discrete Gaussian Free Field $\regmGFF{L}{m}{\epsilon}$ is replaced by the law of the balanced Branching Random Walk $\regmbalBRW{L}{m}{\epsilon}$.

\begin{figure}
\begin{center}
\def\svgwidth{0.7\columnwidth}
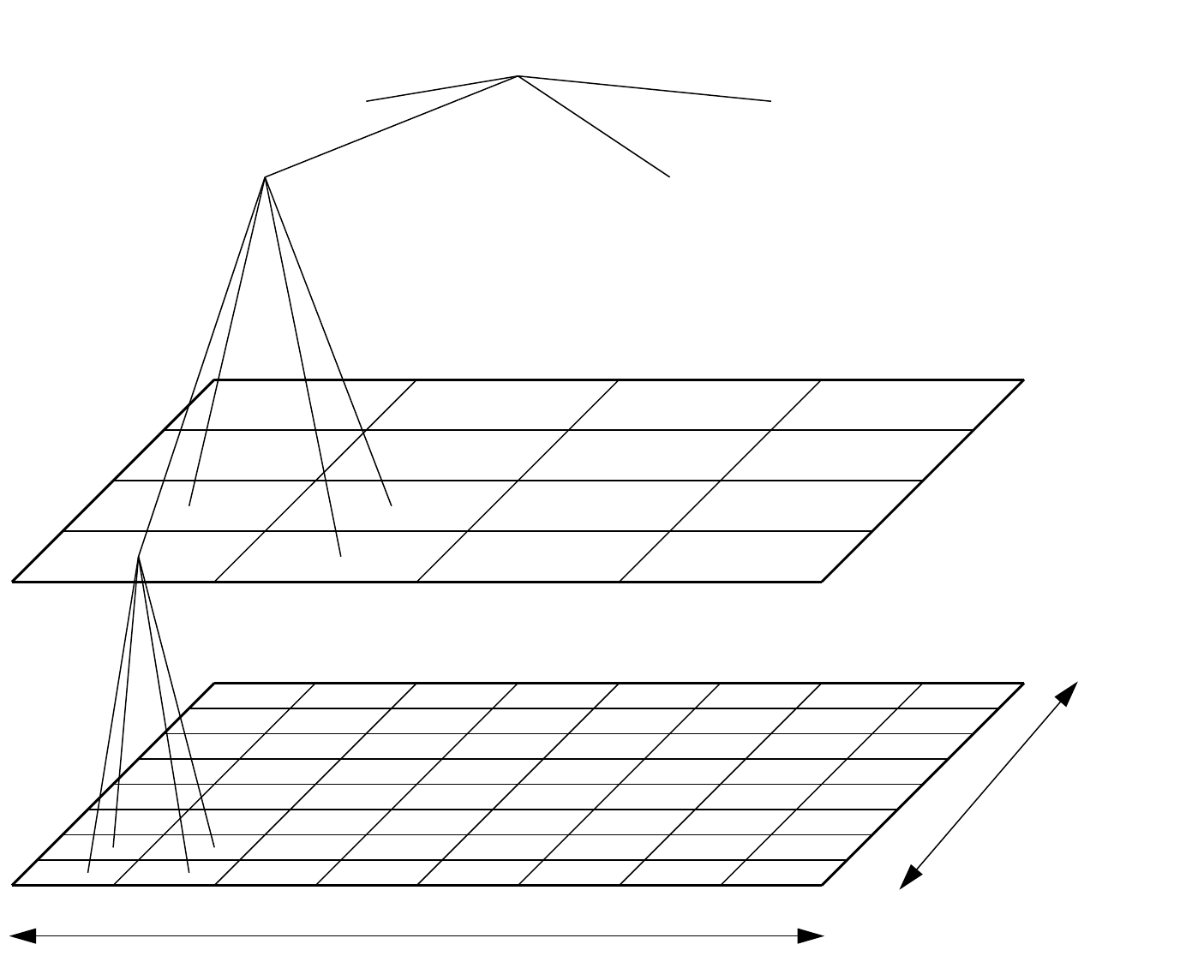
\end{center}
\caption{Branching Random Walk in $\Lambda_{L, \epsilon}$, depicted  for $\epsilon=1/2$ and $L=4$, i.e. $n=1$ and $k=2$. The subtree rooted at $\bar o$ is $\cR_0$, and that rooted at $\bar o_1$ is $\cR_1$.}
\label{fig:embedded-brw-decomposition}
\end{figure}

We now explain how the measure 
$\regmhLv{L}{m}{\epsilon}$ is related to $\PLv{t_k}$. 
For that it is useful write, for  $v\in\cR_k$ with $\horo(v)=n$ descendant of
$w$ with $\horo(w)=0$,
\begin{equation}
\label{eq:embd-balanced-brw-decomposition-bis}
\embdbalBRW_v = \embdbalBRW_{w} +  \embdbalBRW_v^{w}.
\end{equation}
This decomposition of the Branching Random Walk $\embdbalBRW$ is illustrated in Figure \ref{fig:embedded-brw-decomposition}.
With $\embdbalBRWregGMC{L}{\epsilon}{ } \equiv \regGMC{L}{\epsilon}{\gamma}(\embdbalBRW) $ as in \eqref{eq:BRW-GMC-depending-on-box-only} the decomposition \eqref{eq:embd-balanced-brw-decomposition-bis} gives
\begin{align}
\label{eq:approximated-gmc-volume-l-bis}
\embdbalBRWregGMC{L}{\epsilon}{ }
 &= \epsilon^2 \sum_{v \in \Gen{n+k}(\cR_k) } e^{\gamma \embdbalBRW_v - \frac{\gamma^2}{2} (n - \frac{1 - d^{-n}}{d-1}) } 
 = \epsilon^2 \sum_{w \in \Gen{k}(\cR_k) } \sum_{v \in \Gen{n}^w} e^{\gamma ( \embdbalBRW_w + \embdbalBRW_v^w  )  - \frac{\gamma^2}{2}(n - \frac{1 - d^{-n}}{d-1} ) } \nnb
 &=  \sum_{w \in \Gen{k}(\cR_k) } e^{\gamma \embdbalBRW_w }  \epsilon^2 \sum_{v \in \Gen{n}^w } e^{\gamma \embdbalBRW_v^w - (n - \frac{1 - d^{-n}}{d-1})  } 
= \sum_{w \in \Gen{k}(\cR_k) } e^{\gamma \embdbalBRW_w} \embdbalBRWregGMC{1}{\epsilon}{w},
\end{align}
where $(\embdbalBRWregGMC{1}{\epsilon}{w})_{w \in \Gen{k}(\cR_k)}$ is a collection of independent random variables with
\begin{equation}
\label{eq:convergence-MAnn-to-MA-bis}
\embdbalBRWregGMC{1}{\epsilon}{w} \to \MA
\end{equation}
in distribution as $\epsilon \to 0$.

We view $\embdbalBRWregGMC{L}{\epsilon}{}$ as a multiplicative chaos in the domain $[0,L)^2$.
However, unlike for the mul\-ti\-pli\-ca\-tive chaos defined in \eqref{eq:balanced-brw-gmc-recursion},
the sum in \eqref{eq:approximated-gmc-volume-l-bis} is not normalised by $L^2= 4^k$.
Moreover, there is no correction $-\frac{\gamma^2}{2} (k - d^{-n} \frac{1-d^{-k}}{d-1})$ to the values $(\embdbalBRW_w)_{w \in \Gen{k}(\cR_k)}$.
Completing these missing terms in \eqref{eq:approximated-gmc-volume-l-bis}, we have from the recursion \eqref{eq:balanced-brw-gmc-recursion} now applied to $\embdbalBRW$ that
\begin{equation}
\label{eq:gmc-l-vs-gmc-1-bis}
\embdbalBRWregGMC{L}{\epsilon}{ } = 4^k e^{\frac{\gamma^2}{2} (k - d^{-n} \frac{1-d^{-k}}{d-1} )} \big( \frac{1}{4^k} \sum_{w \in \Gen{k}(\cR_k)} e^{\gamma \embdbalBRW_w -\frac{\gamma^2}{2} (k - d^{-n} \frac{1-d^{-k}}{d-1}) }  \balBRWregGMC{1}{\epsilon}{z} \big) 
 = 4^{\frac{\kappa + 1}{\kappa} k} e^{- d^{-n} \frac{1-d^{-k}}{d-1} }  \balBRWregGMC{1}{\epsilon}{}
\end{equation}
where we used \eqref{eq:exponents-first-order} and $e^{\frac{\gamma^2}{2} } = 4^{\frac{\gamma^2}{2\log 4}} =  4^{\frac{1}{\kappa}}$.

Let now $\embdPLv{t_k}$ be the probability measure on $\sigma (\embdbalBRW_v \colon v \in \cL_k)$ as introduced above Corollary \ref{cor:decay-correlation-tree}.
The embedding of $\embdbalBRW$ into subsets of $\mathbb{R}^2$
naturally gives an identification between $\cL_k$ (which here
is viewed as those vertices $v\in \cR_k$ with $\horo(v)=0$)
and integrals of the field over squares of side $1$.
Hence, we can view $\embdPLv{t_k}$ as a probability measure on such integrals.
Likewise we view $\regmhLv{L}{m}{\epsilon}$ as a probability measure on 
either
$(\embdbalBRW_v)_{v \in \cL_k}$ or on such integrals.
We claim that as $\epsilon \to 0$ and then $m\to 0$, we have that $\regmhLv{L}{m}{\epsilon} \to \embdPLv{t_k}$ weakly.
To see this we first note that \eqref{eq:gmc-l-vs-gmc-1-bis} and the convergence \eqref{eq:convergence-MAnn-to-MA-bis} imply that as $\epsilon \to 0$, we have
\begin{equation}
\E_{\regmLBRW{L}{m}{\epsilon}} [ e^{- \lambda \embdbalBRWregGMC{L}{\epsilon}{ }(\phi) } ]
= \E [  e^{- \lambda \embdbalBRWregGMC{L}{\epsilon}{ }(\embdbalBRW) }  ]
= \E [ e^{-t_k \balBRWregGMC{1}{\epsilon}{}} ] \to \E[ e^{- t_k \MA} ] ,
\end{equation}
where the first equality assumes $m$ is small enough so that $m<2^{-L}$, see the text below \eqref{eq:embedded-BRW-avg-box}.

Let now $ B$ be an event depending on finitely many 
$(\embdbalBRW_{v_i})_{i=1,\ldots, N}$ with $\horo(v_i)=0 $, where we choose $k$ large enough so that $v_i \in \cL_k$ for all $i=1,\ldots,N$.
Then, we have for $K>k$ and as $\epsilon \to 0$
\begin{align}
\regmhLv{L}{m}{\epsilon} (B)
&= \frac{\E_{\regmLv{L}{m}{\epsilon}}[ \mathbf{1}_{B} e^{-\lambda \embdbalBRWregGMC{L}{\epsilon}{ }(\phi) } ]}{\E_{\regmLv{L}{m}{\epsilon}}[  e^{-\lambda \embdbalBRWregGMC{L}{\epsilon}{ }(\phi)} ] }
= \frac{\E[ \mathbf{1}_{B} e^{-\lambda \embdbalBRWregGMC{L}{\epsilon}{ }(\embdbalBRW) } ]}{\E[  e^{-\lambda \embdbalBRWregGMC{L}{\epsilon}{ }(\embdbalBRW)} ] }
=  \frac{\E[ \mathbf{1}_{B} e^{- t_k \balBRWregGMC{1}{\epsilon}{} }]}{\E[  e^{-t_k \balBRWregGMC{1}{\epsilon}{} }] } \nnb
&\to 
 \frac{\E[ \mathbf{1}_{B} e^{- t_k \MA} ]}{\E[  e^{-t_k \MA} ] } 
 = \embdPLv{t_k}(B),
\end{align}
from which the claimed convergence follows.

Having explained the relation of $\embdPLv{t_k}$ and the the measure $\regmLv{L}{m}{\epsilon}$  in \eqref{eq:discrete-gff-liouville-density-with-gmc}, we propose the following conjecture.
For a function $f\in C_c^\infty(\R^2)$ we write $f_x(\cdot)=f(\cdot-x)$ for the translation of $f$ by $x\in \R^2$.
\begin{conjecture}
\label{thm:liouville-decay-correlations-test-function}
Let $\regmLv{L}{m}{\epsilon}$ be defined as in \eqref{eq:discrete-gff-liouville-density-with-gmc} and let $\phi \sim \regmLv{L}{m}{\epsilon}$.
There exists a family of events $(\cA_{x,r,L})_{x\in \R^2, r\geq 0, L>0}$ with $\regmLv{L}{m}{\epsilon}(\cA_{x,r,L}) \to 1$ as first $\epsilon \to 0$,
then $m\to 0$ and finally $L\to \infty$,
such that for any $f,g \in C_c^\infty(\R^2)$,
\begin{equation}
\label{eq:decay-hLv-indicator}
\limsup_{r \to \infty} \limsup_{L \to \infty} \limsup_{m\to 0}
\limsup_{\epsilon \to 0} \sup_{y \in \R^2 \colon |x-y|=r}
|\langle \phi(f_x) \phi(g_y) \mathbf{1}_{ \cA_{x,r,L}} \rangle_{ \regmLv{L}{m}{\epsilon} } |  = 0.
\end{equation}
\end{conjecture}
We hope that our methods can be adapted to resolve Conjecture 
\ref{thm:liouville-decay-correlations-test-function}.
More ambitiously, one can hope that the sets $\cA_{x,r,L}$ in its statement can be removed.

\section{Literature}
\label{sec:literature}

The balanced Branching Random Walk was first introduced by Ding and Goswami in \cite{MR3742400} as switched sign Branching Random Walk to study the usual Branching Random Walk in the context of first passage percolation.
Subsequently, this model was used by Roy in \cite{MR4706674} to establish the first and second order asymptotics of the expectation of maximum of a Branching Random Walk under a hard wall constraint,
i.e.\ under the condition that the minimum of the Branching Random Walk is non-negative. 
In contrast to these works, the balanced Branching Random Walk is the main object of interest in the present context.
We are not aware of other works that study the balanced Branching Random Walk as a stand-alone object,
and, to the best of our knowledge, our work is the first, which investigates the relation between the values of a Branching Random Walk model
and small values of its multiplicative chaos.

Our work is related to \cite{Fels2024BRW-1-2} by Fels, Hartung and Louidor, which establishes under the same conditioning as in \cite{MR4706674} sharp results on the limiting distribution of the conditional maximum,
thereby extending the results in \cite{MR4706674}.
We emphasise, however, that our point of view differs from these references,
in the sense that \cite{Fels2024BRW-1-2} and \cite{MR4706674} 
inject information on the values of the Branching Random Walk at level $n$ directly by conditioning the values to be positive at level $n$,
while we infer information on the values of the Branching Random Walk from the fact that the Gaussian multiplicative chaos is small.

The theory of multiplicative chaos, which underlies the definition of the measures $\PLv{t}$, goes back to the work \cite{MR829798} of Kahane,
where this object is constructed for certain log-correlated Gaussian fields and studied as a random measure on a finite domain in the Euclidean space.
For a survey on recent results on the multiplicative chaos, we refer to the review \cite{MR3274356} by Rhodes and Vargas.
In the present context we are only interested in the full mass of the multiplicative chaos, and all we need in our proofs are the recursion \eqref{eq:balanced-brw-gmc-recursion} and basic results for the negative moments of $\MA$.

The exponential moments of the multiplicative chaos, or equivalently the small deviation pro\-ba\-bi\-li\-ties $\P(\MA \leq s)$,
also appear in the construction of the quantum Mabuchi $K$-energy in \cite{MR4383250} by Lacoin, Rhodes and Vargas.
This work considers a Gaussian field on a bounded domain with zero average, which is corresponds to the balanced Branching Random Walk in our case, and thus shares some common ideas with the present work.
For instance, the proof \cite[Theorem 4.5]{MR4383250} is similar to the proof of the upper bound in \eqref{eq:crude-small-ball-asymptotics}.
 
Our work is remotely related to \cite{MR4047550} by Remy on the Gaussian multiplicative chaos of the Gaussian Free Field on the unit circle $S^1$,
in which an exact formula for the moments of the multiplicative chaos is proved.
This even allows to compute the density and the small deviations
probability of the multiplicative chaos explicitly as a fractional power of an exponential random variable,
which readily gives an exact formula for the small deviation probability.
However, this additional information, which is a key ingredient in the proof of Theorem \ref{thm:pt-liouville-decay-l1},
comes at the cost of giving up the recursive identity \eqref{eq:balanced-brw-gmc-recursion}.
Nonetheless, we hope that a variation of our method can yield
a decay of correlations similar to Corollary \ref{cor:decay-correlation-tree} in that case.
In particular, it would be interesting to see, if the formula for the density of the Gaussain multiplicative chaos can be used to establish finer properties,
such as quantitative estimates on the decay of correlation,
which remain completely open from our analysis.

As explained in Section \ref{sec:qft-models}, the motivation of our work ultimately comes from Euclidean field theory
with the measure $\PLv{t_k}$ in \eqref{eq:pt-liouville-definition} with $t_k=\lambda \base^k$ being a hierarchical version of the massless Liouville field theory in infinite volume.
We therefore continue the literature review with references from this area.
The Liouville model was introduced in more generality by H{\o}egh-Krohn in \cite{MR292433} and later studied by Albeverio and H{\o}egh-Krohn  in \cite{MR395578},
as a Euclidean field theory with exponential interaction.
It is thus often referred to as the H{\o}egh-Krohn model. 
These earlier works provide a non-probabilistic construction of the field in the $L^2$ regime for an interaction with space cut-off,
which in our notation corresponds to the box $\Lambda_L$ and $\beta < 4\pi$.
Moreover, it is argued that the space cut-off can be removed,
thereby obtaining a Euclidean theory in infinite volume.
The full $L^1$ regime $\beta < 8\pi$ was later studied by Kusuoka in \cite{MR1190536}. 
Notably, these non-probabilistic constructions were given before the work of Kahane.
Further results on Euclidean field theories with exponential interaction were established in \cite{MR0456220}, \cite{MR3116316} and \cite{MR692312}.

The first probabilistic constructions of a Liouville field theory were achieved using stochastic quantisation techniques by Hoshino, Kawabi and Kusuoka in \cite{MR4238209} \cite{MR4528973}, by Oh, Robert and Wang in \cite{MR4324379}  and by Albeverio, De Vecchi and Gubinelli in \cite{MR4415393}.
This approach allows to construct the Liouville field theory as the stationary measure of a certain SPDE known as the stochastic quantisation equation introduced by Parisi and Wu in \cite{ParisiWu1981PerturbationTheory}. 

A recent construction of the Liouville field theory in infinite volume was given by Barashkov and Gubinelli in \cite{MR4672110} using a variational formulation,
which allows to express the exponential moment of the interaction as a solution to a stochastic control problem.

One common feature of all probabilistic constructions of the Liouville field theory mentioned above is that a massive Gaussian Free Field $\ctmGFF{m}$ on $\R^2$ (or on a torus $\T^2$) with a fixed mass $m>0$ is used as a reference measure.
More precisely, with the same caveats as explained below \eqref{eq:liouville-qft-formal-definition} and translated to our notation,
the target measure $\ctmLv{m}$ in all previous references is given by
\begin{equation}
\label{eq:massive-liouville-theory}
\ctmLv{m}(d\phi) \propto \exp \Big[ - \lambda \int_{\R^2} e^{\sqrt{\beta} \phi (x)} dx \Big] d \ctmGFF{m}(\phi) .
\end{equation}
In particular, the reference measure $\ctmGFF{m}$ is a priori a well-defined Gaussian field on $\R^2$ thanks to the mass term $m$. 
By contrast, our methods are designed to take,
at least in the hierarchical case, the formal limit $m\to 0$,
which cannot be extracted from existing works.
The main difficulty is to deal with the fact that in the limit $m\to 0$ the reference measure becomes ill-defined,
which can be seen from the divergent term $\frac{1}{2\pi}\log \frac{1}{m}$ in \eqref{eq:variance-massive-gff}.
As we suggest in Section \ref{sec:qft-models},
the natural way to construct the massless theory is to start with a massive Gaussian Free Field,
subtract the average on a finite box $\Lambda_{L}$,
and then take the limit $m\to 0$. 
While strictly speaking our methods were crafted for the Branching Random Walk,
we believe an analogous approach works for the Gaussian Free Field.
In particular, it is of interest to establish the tightness of the measures $\ctLmLv{L}{m}$ defined in \eqref{eq:liouville-qft-formal-definition},
thereby giving a rigorous mathematical construction of a massless Liouville field theory in the plane.

There is also a substantial literature on the Liouville conformal field theory on the sphere.
A probabilistic construction of this model was given by Kupiainen, Rhodes and Vargas in \cite{MR3832670}; see also the subsequent works \cite{MR4131028} \cite{MR4060417}.
In this case the reference measure is the Gaussian Free Field on the sphere with zero average, which would correspond to the balanced Branching Random Walk in our work.
The main difference, however, is that the Liouville conformal field theory on the sphere is defined on a compact space, and therefore, the infinite volume limit is not of interest.

A true massless theory in infinite volume was recently constructed by Bauerschmidt and Webb in \cite{MR4767492} for the sine-Gordon field at the free fermion point $\beta = 4\pi$ by establishing a relation between the sine-Gordon field and the massive Thirring model on the level of correlation functions.
This allows then to take a formal limit $m\to 0$,
thereby giving a rigorous construction of the massless theory.
In fact, our formulation of the weak decay of correlations for the hierarchical Liouville model in Corollary \ref{cor:decay-correlation-tree} resembles \cite[Theorem 1.3]{MR4767492},
but we are not able to quantify the decay.

Finally, since we also touch on the negative exponential moments of $\frac{1}{2}(\Mplus + \Mminus)$ in Theorem \ref{thm:sinh-mgf-asymptotics},
we give references to the sinh-Gordon Euclidean field theory,
which is the measure defined by
\begin{equation}
\label{eq:massive-sinh-theory}
\nu_m^{\ShG} (d\phi) \propto \exp \Big[ - \lambda \int_{\R^2} \cosh( \sqrt{\beta} \phi (x) )dx \Big] d\nu_m^{\GFF}(\phi) .
\end{equation}
Thus, compared to \eqref{eq:massive-liouville-theory} the exponential function is replaced by a hyperbolic cosine,
which, after renormalisation by Wick ordering, leads to two non-independent Gaussian multiplicative chaoses $\Mplus$ and $\Mminus$,
where $\Mplus$ is constructed from $\phi$ and $\Mminus$ is constructed from $-\phi$.
In fact, the general Euclidean field theory with exponential interaction introduced in \cite{MR292433} also covers the sinh-Gordon model,
and thus, the non-probabilistic construction in this reference of the infinite volume limit also applies to the sinh-Gordon field theory.
A probabilistic construction as a measure of the form \eqref{eq:massive-sinh-theory} in infinite volume and for fixed mass $m>0$ was recently given by Barashkov and De Vecchi in \cite{BarashkovDeVecchi2021Elliptic} for the full $L^2$ regime of the multiplicative chaos,
i.e.\ for $\beta <4\pi$, using elliptic stochastic quantisation techniques.
However, the construction of a massless theory remains an open problem.
In contrast to the Liouville field theory, it should be possible to take the formal limit $m\to 0$ without removing the average,
at least in the $L^2$ regime,
and given that the reference measure does not exist in the limit, this seems to be an interesting and challenging problem. 

Notable progress in this direction was recently reported by Guillarmou, Gunaratnam and Vargas in \cite{Guillarmou2024sinhGordon},
where they are able to give a construction of the massless sinh-Gordon field on the infinite cylinder.
Our analysis suggests that, in order to construct this model in the plane,
it is essential to have precise estimates on the asymptotics of the negative exponential moments up to the second order as $m\to 0$.
In the case of the Branching Random Walk, this corresponds to the asymptotics of $\tilde \Lambda^u$ defined in \eqref{eq:lower-bound-sinh-imbalance}; see also the discussion below \eqref{eq:pt-sinh-definition}.
A first result in this direction was achieved by Barashkov, Oikarinen and Wong in \cite{Barashkov2024SmallDeviations}, where they prove
\begin{equation}
\label{eq:bounds-sinh-gordon-free-energy}
c L^{2} \leq -\log \E_{\bar \nu_{m}^\GFF } [e^{-\lambda \int_{\Lambda_L}\cosh(\sqrt{\beta}\phi(x))dx }] \leq C L^{2}
\end{equation}
for some constants $0< c\leq C$, which are independent of $m>0$.
Here, $\bar \nu_{m}^\GFF$ denotes the Gaussian Free Field with removed zero mode.
This result is the analogue to our Lemma \ref{lem:sinh-mgf-asymptotics-bounds},
though our result is proved for the simplified setting,
in which Branching Random Walk takes the role of the Gaussian Free Field.
Nonetheless, we believe that some of our ideas apply also apply to the more general case, 
and it would be interesting to see, if \eqref{eq:bounds-sinh-gordon-free-energy} can be established in an asymptotic sense  with $c=C(1+o_L(1))$.

\section{Asymptotics of the negative exponential moments}
\label{sec:large-deviations}

In this section we present the proofs of Theorem \ref{thm:liouville-mgf-asymptotics} and Theorem \ref{thm:sinh-mgf-asymptotics}.
We are mainly concerned with the negative exponential moments of $\MA$.
The case $\frac{1}{2}(\Mplus + \Mplus)$ is similar and is presented in Section \ref{ssec:sinh-mgf-asymptotics}.

\subsection{Exponential moments of $\MA$: proof of Theorem \ref{thm:liouville-mgf-asymptotics}}
\label{ssec:liouville-mgf-asymptotics}

We first establish the following a priori estimates for the exponential moments of $\MA$.
\begin{lemma}
\label{lem:liouville-mgf-asymptotics-bounds}
Let $\alpha$ be as in \eqref{eq:exponents-first-order}. There exist constants $0< c \leq C$ and $t_0\geq 0$, such that for $t \geq t_0$
\begin{equation}
\label{eq:liouville-mgf-asymptotics-bounds}
e^{-C  t^\alpha} \leq \E[e^{-t \MA}] \leq e^{-c t^\alpha}.
\end{equation}
\end{lemma}

It is worthwhile commenting on the appearance of the exponent $\alpha$ in \eqref{eq:liouville-mgf-asymptotics-bounds}: As we explain below, the computation of the Laplace transform of $M^A$ is based on small ball estimates for $M^A$, see
\eqref{eq:crude-small-ball-asymptotics}. As our proof of the lower bound in the latter shows, a good strategy for computing the small ball
probability 
is to force the variables in the first $\ell$ generations of the tree to be all small. The probability cost of this effect is exponential in the
volume of this subtree,  while the effect on $M^A$ is proportional to $e^{-\gamma^2 \ell/2}$. Optimizing over $\ell$ as function 
of $s$ in \eqref{eq:crude-small-ball-asymptotics} then yields the constant $\kappa$ in the latter, which then translates, using Laplace's method,  to $\alpha$
in \eqref{eq:liouville-mgf-asymptotics-bounds}. The complementary upper bound in \eqref{eq:crude-small-ball-asymptotics}, on the other hand, is an easy application of the hierarchical structure of the Branching Random Walk, Jensen's inequality, and a crude a priori estimate on negative exponential moments of $M^A$.

\begin{proof}
We show that for some constants $0<c'\leq C'$ and as $s \to 0$,
\begin{equation}
\label{eq:crude-small-ball-asymptotics}
e^{-C'\frac{1}{s^\kappa}} \leq \P(\MA \leq s) \leq e^{-c' \frac{1}{s^\kappa}},
\end{equation}
where $\kappa$ is as in \eqref{eq:exponents-first-order}. From this result the asymptotics in \eqref{eq:liouville-mgf-asymptotics-bounds}
follow by writing
\begin{equation}
\label{eq:liouville-mgf-by-laplace-method}
\E[e^{-tM^A}]=\int_0^\infty t e^{-ts} \P(M^A\leq s) ds\leq  O(t e^{-ct})+t \int_0^c e^{-st} e^{-c's^{-\kappa}}ds,
\end{equation}
and noting that by Laplace's method the main contribution to the integral is given by a neighborhood of $s=\mathfrak{c}t^{-1/(\kappa+1)}$ 
with $\mathfrak{c}=(\kappa c')^{1/(\kappa+1)}$, in which case the right hand side of \eqref{eq:liouville-mgf-by-laplace-method} is bounded above by 
$e^{-(\mathfrak{c}+c'/\mathfrak{c}^\kappa) t^\alpha(1+o_t(1))}$. The argument for the lower bound is similar.

To obtain the upper bound in \eqref{eq:crude-small-ball-asymptotics},
we let $\ell \in \N$ be such that $s= r e^{-\frac{\gamma^2}{2}\ell}$ for some $r<1$, which is determined below.
Then we observe that by \eqref{eq:balanced-brw-gmc-recursion},  Jensen's inequality and the fact that $\sum_{v\in \Gen{\ell}} A_v = 0$, we have
\begin{equation}
\MA \stackrel{d}{=} \frac{e^{-\frac{\gamma^2}{2}\ell}}{d^\ell} \sum_{v\in \Gen{\ell}} e^{\gamma A_v} \MA_v
\geq  e^{-\frac{\gamma^2}{2}\ell} e^{\frac{1}{d^\ell}\sum_{v\in \Gen{\ell}}   \log \MA_v }.
\end{equation}
Then we have by Markov's inequality for any $\theta>0$
\begin{align}
\label{eq:upper-bound-small-ball-by-markov}
\P(\MA \leq s) &\leq  \P( \frac{1}{d^\ell} \sum_{v\in \Gen{\ell}} \log \MA_v \leq \log r)
= \P \big( \sum_{v\in \Gen{\ell}} \log \MA_v \leq d^\ell \log r \big)  \nnb
&= \P\big( \prod_{ v\in \Gen{\ell} } (\MA_v)^{-\theta} \geq r^{-\theta d^\ell } \big)
\leq \E  \big[ (\MA)^{-\theta} \big]^{d^\ell} r^{\theta d^\ell} .
\end{align}
By 
\cite[Theoreme IIA]{MR1717530} or \cite[Theorem 4a]{MR1400758}
we have that $e^{c_\theta} \equiv \E[( \MA)^{-\theta} ] < \infty$ for some constant $c_\theta>0$ depending on $\theta$.
Therefore, it follows from the previous display that
\begin{equation}
\P(\MA \leq s) \leq e^{  c_\theta d^\ell  + d^\ell \theta \log r } = e^{  -( \theta \log \frac{1}{r} - c_\theta ) d^\ell}.
\end{equation}
Now, for a fixed $\theta>0$ and for $r$ sufficiently close to $0$, the prefactor in the exponent is negative.
Noting that $d^{\ell} = ({r}/{s})^{\kappa}$, the upper bound in \eqref{eq:crude-small-ball-asymptotics} follows.

To prove the lower bound in \eqref{eq:crude-small-ball-asymptotics},
we use the following explicit construction of $(Y_e)_{e\in \cT}$ in \eqref{eq:def-brw-balanced}.
Let $(Z_e)_{e\in \cT}$ be a collection of independent centred Gaussian random variables with variance $d/(d-1)$ and set
\begin{equation}
\label{eq:explicit-construction-Y}
Y_e = Z_e - \frac{1}{d} \sum_{e'\in E(v)} Z_{e'}.
\end{equation}
Then it is easy to check that $\E[ Y_e^2] = 1$ and $\E[Y_e Y_{e'}] = -1/(d-1)$ as needed.
For $u>0$ we define the events
\begin{align}
\label{eq:good-event-lower-bound-small-ball}
\cA_{\ell,u} &= \big\{ \forall i=0,\ldots, \ell-1 \colon |Y_e| \leq \frac{u}{2^{\ell-i+1}} \text{ for all } e \in \Gen{i} \big\}, \nnb
\tilde \cA_{\ell, u} &= \big\{ \forall i=0,\ldots, \ell-1 \colon |Z_e| \leq \frac{u}{2^{\ell-i+1}} \text{ for all } e \in \Gen{i} \big\},
\end{align}
where $e\in \Gen{i}$ means that $e\in E(v)$ for some $v\in \Gen{i}$.
Moreover, we define
\begin{equation}
\label{eq:event-infty-norm-bounded}
\cB_{\ell, u} = \{  \lpnorm{A}{\ell}{\infty} \leq u \}.
\end{equation}
Note that on the event $\cA_{\ell, u}$, we have for any $v\in \Gen{\ell}$
\begin{equation}
|A_v| \leq \sum_{e \in o\leftrightarrow v} |Y_e| \leq u \sum_{i=0}^{\ell -1} 2^{-(\ell -i+1)} = u \sum_{j=0}^{\ell -1} 2^{-j} = u(1-2^\ell ) \leq u,
\end{equation}
which shows that $\cA_{\ell, u} \subseteq \cB_{\ell, u}$.
Moreover, we have by \eqref{eq:explicit-construction-Y} and the triangle inequality $\tilde \cA_{\ell,u/2} \subseteq \cA_{\ell,u}$.
Thus, in order to find a lower bound $\P(\cB_{\ell, u})$,
it suffices to establish a lower bound on $\P(\tilde \cA_{\ell,u/2})$.

Note that for $i\in \{0,\ldots,\ell-1\}$ there are $d^i$ edges between level $i$ and level $i+1$. 
Further note that if $Z$ is a Gaussian random variable with variance $d/(d-1)$, then there exists $c_d>0$, such that for $r\leq 1$
\begin{equation}
\label{eq:gaussian-small-ball}
\P(|Z|  \leq r ) \geq c_d r.
\end{equation}
Using this fact together with the independence of the collection $(Z_e)_{e\in \cT_\ell}$,
we have
\begin{align}
\label{eq:lower-bound-event-infty-norm}
\P(\tilde \cA_{\ell, u}) &=\prod_{i=0}^{\ell-1} \prod_{e\in \Gen{i}} \P\Big(|Z_e| \leq \frac{u}{2^{\ell-i+1}} \Big) 
= \prod_{i=0}^{\ell-1} \P\Big(|Z_e| \leq \frac{u}{2^{\ell-i}} \Big)^{ d^i }  \nnb
& \geq \prod_{i=0}^{\ell-1} \Big( c_d \frac{u}{2^{\ell-i+1}} \Big)^{d^i}
= (uc_d)^{\sum_{i=0}^{\ell -1} d^i  }  2^{-\sum_{i=0}^{\ell -1}(\ell -i+1)d^i} 
\geq {e^{-C_u d^\ell} } ,
\end{align}
where $C_u>0$ is a constant depending on $u$ and $d$ with $C_u \to \infty$ as $u\to 0$.
Note that in the last line of \eqref{eq:lower-bound-event-infty-norm} we used that $\sum_{i=0}^{\ell-1} (\ell-i+1)d^i \leq C_d d^\ell$ for some constant $C_d$,
which is independent of $\ell\in\N$.
We now set $s= Re^{-\frac{\gamma^2}{2} \ell}$ for some $R>1$ to be determined below (after \eqref{eq:lower-bound-small-ball-good-event}).
Note that on the event $\cB_{\ell,u}$ we have
\begin{equation}
\MA \stackrel{d}{=} \frac{1}{d^\ell}  \sum_{v\in \Gen{\ell}} e^{\gamma A_v - e^{-\frac{\gamma^2}{2} \ell} } \MA_v
\leq  e^{\gamma u -\frac{\gamma^2}{2} \ell} \frac{1}{d^\ell} \sum_{v\in \Gen{\ell}} M_v ,
\end{equation}
and thus, we have by the independence of $(\MA_v)_{v\in \Gen{\ell}}$ and $\cB_{\ell, u}$
\begin{align}
\label{eq:lower-bound-small-ball-good-event}
\P(\MA\leq s)&\geq \P(\MA \leq s, \, \cB_{\ell, u})
\geq \P \big( \frac{1}{d^\ell} e^{\gamma u} \sum_{v\in \Gen{\ell}} M_v \leq R, \, \cB_{\ell,u} \big) \nnb
&= \P\big( \frac{1}{d^\ell} e^{\gamma u} \sum_{v\in \Gen{\ell}} \MA_v \leq R \big ) \P(\cB_{\ell,u} ) .
\end{align}
We now choose $R$ and $u$ such that $e^{-\gamma u} R >1$.
Then, by the law of large numbers and the fact that $\E[\MA] =1$ we have
\begin{equation}
\P\big( \frac{1}{d^\ell} e^{\gamma u} \sum_{v\in \Gen{\ell}} \MA_v \leq R \big ) \to 1
\end{equation}
as $\ell \to \infty$.
In particular, there is $\ell_0\in \N$, such that for this choice of $R$ and $u$, we have
\begin{equation}
\label{eq:MLv-lower-bound-by-LLN}
\P\big( \frac{1}{d^\ell} e^{\gamma u} \sum_{v\in \Gen{\ell}} \MA_v \leq R \big ) \geq \frac{1}{2} , \qquad \ell \geq \ell_0.
\end{equation}
Since $d^{\ell}= ( {R}/{s})^{\kappa}$, the lower bound in \eqref{eq:crude-small-ball-asymptotics} follows by \eqref{eq:lower-bound-event-infty-norm}.
\end{proof}

We now turn to the proof of Theorem \ref{thm:liouville-mgf-asymptotics},
for which we rewrite the negative exponential moment of $\MA$, so that arguments from large deviations theory apply.
Specifically, we use an approximate sub-additive structure of large deviations probabilities, to prove a priori a weak large deviations principle (Proposition \ref{prop:weak-ldp}) for averages of functions of $A_v$ with $v\in D_k$ and $k$ large, see \eqref{eq:definition-ld-rv} and the decomposition \eqref{eq:mgf-as-ld}; a technical point that needs  then to be addressed is that we only get 
this way a weak large deviations principle (for good reasons - the rate function vanishes at $+\infty$), and we use the a priori moment bounds of Lemma \ref{lem:liouville-mgf-asymptotics-bounds} in order
to restrict attention to compact subsets of $\R_+$, for which the weak large deviations principle suffices.

Proceeding with the proof, we recall that $t_k = \lambda \base^k$, $k\in \N$ and $\lambda>0$,
and expand $\MA$ at level $k$ according to \eqref{eq:balanced-brw-gmc-recursion} to obtain
\begin{equation}
\E[e^{-t_k \MA} ]=\E[e^{-\lambda \base^k\MA}]
= \E\big[ e^{-\lambda \sum_{v\in \Gen{\ldlv}} e^{\gamma A_v} \MA_v } \big],
\end{equation}
where $(\MA_v)_{v\in \Gen{\ldlv}}$ is a collection of independent identically distributed random variables with $\MA_v \stackrel{d}{=} \MA$.
Conditioning on $(A_v)_{v\in \Gen{\ldlv} }$ and using the independence of $(\MA_v)_{v\in \Gen{\ldlv}}$
we can write the last display as
\begin{equation}
\label{eq:mgf-as-ld}
\E[e^{-\lambda \sum_{v\in \Gen{\ldlv}} e^{\gamma A_v} \MA_v } ]\! =\!
\E \Big[ \prod_{v\in \Gen{\ldlv} } \E\big[ e^{-\lambda e^{\gamma A_v} \MA_v} \bigm | (A_v)_v  \big] \Big] 
\!= \!\E\Big[ e^{ \sum_{v\in  \Gen{\ldlv} } \log \E \big[ e^{-\lambda e^{\gamma A_v} \MA_v} \bigm | (A_v)_v \big ]   }  \Big] 
\!=\!\E [e^{-d^\ldlv H_\ldlv } ] ,
\end{equation}
where the random variables $H_\ldlv$ are defined by
\begin{equation}
\label{eq:definition-ld-rv}
H_\ldlv 
= \frac{1}{d^\ldlv} \sum_{v\in \Gen{\ldlv}} G(A_v)
\end{equation}
with $G\colon \R\to \R$, 
\begin{equation}
\label{eq:definition-G}
G(x) = -\log \E[e^{-\lambda e^{\gamma x } \MA } ].
\end{equation}
Note that both $G$ and $H_k$ depend on $\lambda>0$,
but since this parameter is fixed throughout this section,
we omit it in the notation.

In what follows, we prove that the laws of $H_\ldlv$ satisfy a (weak) large deviations principle.
The first result in this direction is the following estimate for the function $G$.
Here and below, we write $B(x, \delta)= \{ y\in \R\colon |x-y| < \delta\}$ for the Euclidean open ball in $\R$ with radius $\delta>0$ and centred at $x\in \R$.

\begin{lemma}
\label{lem:continuity-G}
Let $G$ be as in \eqref{eq:definition-G} and let $y\in B(x,\delta)$ for $x\in \R$ and $\delta>0$. Then
\begin{equation}
\label{eq:continuity-G}
G(x) e^{-\gamma \delta} \leq  G(y) \leq  G(x)e^{\gamma \delta}.
\end{equation}
\end{lemma}

\begin{proof}
Since $\MA$ is non-negative, $G$ is non-negative and increasing in $x$.
Therefore, we have 
\begin{equation}
\label{eq:g-increasing}
G(x-\delta ) \leq G(y) \leq G(x + \delta), \qquad y\in B(x,\delta).
\end{equation}
Since $e^{\gamma \delta} \geq 1$, the map $r\mapsto r^{e^{\gamma \delta}}$ is convex
and hence,
Jensen's inequality implies
\begin{equation}
\E[e^{-\lambda e^{\gamma (x+ \delta) }  \MA} ] = 
\E[\big(e^{-\lambda e^{\gamma x } \MA} \big)^{e^{\gamma \delta}} ]
\geq \Big(\E[e^{-\lambda e^{\gamma x } \MA} ] \Big)^{e^{\gamma\delta } } ,
\end{equation}
from which we obtain by taking the logarithm
\begin{equation}
G(x+\delta) \leq e^{\gamma \delta} G(x).
\end{equation}
Together with \eqref{eq:g-increasing}, this proves the upper bound in \eqref{eq:continuity-G}.
The lower bound in \eqref{eq:continuity-G} is true by the same argument, now using Jensen's inequality together with the concavity of $r\mapsto r^{e^{-\gamma \delta}}$.
\end{proof}

For a measurable set $A\subseteq \R$, we define 
\begin{equation}
\label{def-fg}
g(k, A) = \P(H_k \in A)\quad \mbox{\rm  and}  \quad f(k, A) = \log g(k, A).
\end{equation}
When $A= B(x,\epsilon)$ for $x\in \R$ and $\epsilon>0$,
we also write $g(k,\epsilon) \equiv g(k, B(x, \epsilon) )$ and $f(k, \epsilon) \equiv f(k, B(x,\epsilon))$ provided the argument does not involve changing the point $x\in \R$.
The following result establishes submultiplicativity for $g(k, \epsilon)$.

\begin{lemma}
\label{lem:subadditivity-empirical-measure}
For any $x\geq 0$ and $\epsilon >0$ there is $\tilde \epsilon>0$ and a constant $c_{x,\epsilon}>0$ depending on $x$ and $\epsilon$ only such that, for all
$\ldlv, \ldlv_0 \in \N$ with $\ldlv \geq \ldlv_0$, we have
\begin{equation}
\label{eq:subadditivity-empirical-measure}
g(\ldlv, \epsilon) \geq g(\ldlv - \ldlv_0, \tilde \epsilon)^{d^{\ldlv_0}} e^{-c_{x,\epsilon} d^{\ldlv_0} }.
\end{equation}
\end{lemma}

\begin{proof}
Recall that for $v\in \Gen{\ldlv}$ we denote by $\anc{v}{\ldlv-\ldlv_0} \in \Gen{\ldlv_0}$ the unique ancestor of $v$ in level $\ldlv_0$
and by $\cT_{\ldlv-\ldlv_0}^w$ the subtree of depth $\ldlv-\ldlv_0$ starting from $w\in \Gen{\ldlv_0}$.
Let $(A^w)_{w\in \Gen{\ldlv_0}}$ be a collection of $d^{\ldlv_0}$ independent balanced Branching Random Walks.
Then we have
\begin{equation}
(A_v)_{v\in \Gen{\ldlv}} \stackrel{d}{=} ( A_{\anc{v}{\ldlv-\ldlv_0}} + A_v^{\anc{v}{\ldlv-\ldlv_0} } )_{ v\in \Gen{\ldlv} } .
\end{equation}
Since $A_{\anc{v}{\ldlv-\ldlv_0}} = A_{\anc{v'}{\ldlv-\ldlv_0}}$ for $v,v' \in \Gen{\ldlv-\ldlv_0}^w$,
it follows that
\begin{equation}
\label{eq:H-equality-distribution}
H_\ldlv \stackrel{d}{=} 
\frac{1}{d^{\ldlv_0}} \sum_{w \in \Gen{\ldlv_0} } \frac{1}{d^{\ldlv-\ldlv_0} } \sum_{v \in \Gen{\ldlv-\ldlv_0}^w } G(A_w + A_v^w).
\end{equation}
In the rest of the proof we implicitly assume that $H_\ldlv$ is realised through \eqref{eq:H-equality-distribution}.
Let now $\cB_{\ldlv_0,  \delta} = \{ \lpnorm{A}{\ldlv_0}{\infty} \leq \delta \}$ with $\delta>0$ to be determined below in \eqref{eq:determine-delta-subadditivity}.
Then we have by Lemma \ref{lem:continuity-G} on this event
\begin{equation}
\label{eq:hk-lower-upper-bound}
e^{-\gamma \delta } \Big(   \frac{1}{d^{\ldlv_0} } \sum_{w\in \Gen{\ldlv_0}}  H_{ \ldlv-\ldlv_0 }^w  \Big)\leq  
H_\ldlv 
\leq  e^{ \gamma \delta } \Big(  \frac{1}{d^{\ldlv_0} } \sum_{w\in \Gen{\ldlv_0}} H_{\ldlv-\ldlv_0}^w \Big),
\end{equation}
where we denote
\begin{equation}
H_{\ldlv-\ldlv_0}^w =  \frac{1}{d^{\ldlv-\ldlv_0}} \sum_{v\in \Gen{\ldlv-\ldlv_0}^w }  G(A_v^w).
\end{equation}
Note that $(H_{\ldlv-\ldlv_0}^w)_{w\in \Gen{\ldlv_0}}$ are independent and identically distributed with $H_{\ldlv-\ldlv_0}^w \stackrel{d}{=} H_{\ldlv-\ldlv_0}$.
It follows from \eqref{eq:hk-lower-upper-bound} that on the event
\begin{equation}
\{ \forall w\in \Gen{\ldlv_0} \colon H_{\ldlv-\ldlv_0}^w \in B (x, \tilde \epsilon) \} \cap \cB_{\ldlv_0, \delta },
\end{equation}
with $\tilde\epsilon>0$ to be determined below in \eqref{eq:determine-delta-subadditivity}, we have
\begin{equation}
e^{-\gamma \delta } (x-\tilde \epsilon) \leq H_\ldlv \leq e^{ \gamma \delta } (x+ \tilde \epsilon).
\end{equation}
Hence, we have on this event
\begin{equation}
-x(1-e^{-\gamma \delta } )  - e^{-\gamma \delta }  \tilde \epsilon \leq H_\ldlv - x
\leq x ( e^{ \gamma \delta }-1) + e^{\gamma \delta } \tilde \epsilon.
\end{equation}
We now choose $\delta>0$ and $\tilde \epsilon>0$ depending on $x\geq 0$ and $\epsilon>0$, such that
\begin{equation}
\label{eq:determine-delta-subadditivity}
x ( e^{ \gamma \delta }-1) + e^{\gamma \delta } \tilde \epsilon\leq \epsilon ,
\end{equation}
which then implies $x(1-e^{-\gamma \delta } ) +  e^{-\gamma \delta }  \tilde \epsilon \leq \epsilon$.
Thus, we obtain
\begin{equation}
\{ \forall w\in \Gen{\ldlv_0} \colon H_{\ldlv-\ldlv_0}^w \in B (x, \tilde \epsilon)   \} \cap \cB_{\ldlv_0, \delta} \subseteq 
\{  H_\ldlv \in B(x, \epsilon)  \}.
\end{equation}
Using the independence of $(H_{\ldlv-\ldlv_0}^w)_{w\in \Gen{\ldlv_0}}$ and $(A_w)_{w\in \Gen{\ldlv_0}}$ together with the lower bound \eqref{eq:lower-bound-event-infty-norm}, we then have
\begin{equation}
\P(H_\ldlv \in B(x, \epsilon ) ) \geq \P (H_{\ldlv-\ldlv_0} \in B(x, \tilde \epsilon) )^{d^{\ldlv_0}} e^{-c_{x,\epsilon} d^{\ldlv_0}}
\end{equation}
for a constant $c_{x,\epsilon}>0$ depending on $\delta$ and hence on $x$ and $\epsilon$.
\end{proof}

With the subadditivity at hand, we now establish the large deviations principle for $(H_k)_{k\in \N}$.
We first prove the following technical result, which gives rise to the definition of the rate function.

\begin{lemma}
\label{lem:log-probability-comparison}
For every $\epsilon>0$ and $x\geq 0$ there is $\tilde \epsilon >0$ such that
\begin{equation}
\liminf_{k\to \infty} \frac{f(k, B(x,  \epsilon) )}{d^k} \geq \limsup_{k\to\infty} \frac{f(k, B(x, \tilde \epsilon))}{d^k}.
\end{equation}
\end{lemma}

\begin{proof}
For $x \geq 0$ and $\epsilon>0$ given, let $\tilde \epsilon>0$ and $c_{x,\epsilon}$ be as in Lemma \ref{lem:subadditivity-empirical-measure}.
Moreover, we denote
\begin{equation}
a=\liminf_{k\to \infty} \frac{f(k, B(x, \epsilon) )}{d^k}, \qquad b= \limsup_{k\to\infty} \frac{f(k, B(x, \tilde \epsilon))}{d^k}
\end{equation}
and let $(k_l)_{l \in \N}$ and $(k_l')_{l\in \N}$ be sequences such that
\begin{equation}
a =\lim_{l\to \infty} \frac{f(k_l, B(x, \epsilon))}{d^{k_l}} , \qquad b = \lim_{l\to \infty} \frac{f(k_l', B(x, \tilde \epsilon))}{d^{k_l'}} .
\end{equation}
After extracting suitable subsequences, we may assume that $k_l-k_l' \to \infty$.
From Lemma \ref{lem:subadditivity-empirical-measure} with $k=k_l$ and $k_0 = k_l-k_l'$, we then get
\begin{equation}
\frac{f(k_l, B(x, \epsilon) ) }{d^{k_l}} \geq \frac{f(k_l', B(x, \tilde \epsilon))}{d^{k_l'}} - \frac{c_{x,\epsilon}}{d^{k_l'}}. 
\end{equation}
Taking $l\to \infty$, we arrive at $a\geq b$.
\end{proof}

\begin{proposition}
\label{prop:weak-ldp}
Let $\cA = \{ B(x,\epsilon)\cap \R_0^+ \colon x \geq  0, \, \epsilon>0\} $ and define $I \colon \R_0^+ \to [0,\infty]$,
\begin{equation}
\label{eq:rate-function}
I(x) = \sup_{A\in \cA \colon x \in A} - \liminf_{k\to \infty}  \frac{1}{d^k} f(k, A).
\end{equation}
Then,
\begin{equation}
\label{eq:weak-ldp-sufficient}
I(x) = \sup_{A\in \cA \colon x \in A} - \limsup_{k\to \infty}  \frac{1}{d^k} f(k, A).
\end{equation}
In particular, the laws of $(H_k)_k$ satisfy the weak large deviations principle with rate function $I$.
\end{proposition}

\begin{proof}
In light of \cite[Theorem 4.1.11]{MR1619036} we need to show that
\begin{equation}
\sup_{A\in \cA \colon x \in A} - \liminf_{k\to \infty}  \frac{1}{d^k} f(k, A)
= \sup_{A\in \cA \colon x \in A} - \limsup_{k\to \infty}  \frac{1}{d^k} f(k, A).
\end{equation}
Assume that this is not the case. Then there is $x \in \R_0	^+$ and $\delta>0$, such that
\begin{equation}
\label{eq:ldp-strict-inequality}
\sup_{A\in \cA \colon x \in A} - \liminf_{k\to \infty}  \frac{1}{d^k} f(k, A)
\geq \sup_{A\in \cA \colon x \in A} - \limsup_{k\to \infty}  \frac{1}{d^k} f(k, A) +\delta .
\end{equation}
Let $A_{\delta/2} \in \cA$ be such that it approximates the supremum on the left hand side of \eqref{eq:ldp-strict-inequality} up to precision $\delta/2$, i.e.\ $x\in A_{\delta/2}$ and
\begin{equation}
\sup_{A\in \cA \colon x \in A} - \liminf_{k\to \infty}  \frac{1}{d^k} f(k, A) \leq - \liminf_{k\to \infty}  \frac{1}{d^k} f(k, A_{\delta/2}) + \delta/2 ,
\end{equation}
so that we have from \eqref{eq:ldp-strict-inequality}
\begin{equation}
\label{eq:proving-weak-ldp-comparision-liminf-limsup}
- \liminf_{k\to \infty}  \frac{1}{d^k} f(k, A_{\delta/2})
\geq
\sup_{A\in \cA \colon x \in A} - \limsup_{k\to \infty}  \frac{1}{d^k} f(k, A) +\delta/2 .
\end{equation}
Note that we may assume that $A_{\delta/2} = B(x,\epsilon)\cap [0,\infty)$ for some $\epsilon>0$. 
By Lemma \ref{lem:log-probability-comparison}, we have that there exists $\tilde \epsilon>0$, such that
\begin{equation}
\limsup_{k\to\infty} f(k, B(x, \tilde \epsilon ) ) \leq \liminf_{k\to \infty} f(k, B(x,\epsilon) ).
\end{equation}
Thus, replacing the supremum on the right hand side of \eqref{eq:proving-weak-ldp-comparision-liminf-limsup} by just $B(x, \tilde \epsilon)$, we have 
\begin{equation}
- \limsup_{k\to \infty}  \frac{1}{d^k} f(k,  B(x,\tilde \epsilon))
\geq 
- \limsup_{k\to \infty}  \frac{1}{d^k} f(k,  B(x,\tilde \epsilon ) ) + \delta/2 ,
\end{equation}
which is a contradiction. Thus, \eqref{eq:weak-ldp-sufficient} holds.

Since $\cA$ is a base of the standard topology on $\R_+$, we have by \cite[Theorem 4.1.11]{MR1619036} that the laws of $(H_k)_{k}$ satisfy the weak large deviations principle with rate function $I$. 
\end{proof}

\begin{remark}
By definition (see \cite[Page 7]{MR1619036}, the weak large deviations principle is equivalent to the following upper and lower bounds:
\begin{enumerate}
\item 
for all compact sets $F \subseteq \R$, we have
\begin{equation}
\label{eq:ld-rate-function-compact-sets}
\limsup_{k\to \infty} \frac{1}{d^k} \log \P( H_k \in F) \leq -\inf_{x\in F} I(x) ,
\end{equation}
\item for all open sets $G\subseteq \R$, we have
\begin{equation}
\label{eq:ld-rate-function-open-sets}
\liminf_{k\to \infty} \frac{1}{d^k} \log \P( H_k \in F) \geq -\inf_{x\in G} I(x) .
\end{equation}
\end{enumerate}
\end{remark}
The representation of the negative exponential moments of $\MA$ in \eqref{eq:mgf-as-ld} together with the weak large deviations principle for $(H_k)_{k}$ put us in the situation similar to Varadhan's lemma as stated for instance in \cite[Theorem 4.3.1]{MR1619036}.
In our case the rate function $I$ may not have compact level sets, which we overcome with the a priori bounds of Lemma \ref{lem:liouville-mgf-asymptotics-bounds}.

\begin{proof}[Proof of Theorem \ref{thm:liouville-mgf-asymptotics}]
To prove the convergence \eqref{eq:liouville-mgf-asymptotics}, it suffices to prove that, as $\ldlv \to \infty$, 
\begin{equation}
\label{eq:ld-convergence}
-\frac{1}{d^k} \log \E[e^{-d^k H_k}] \to 
h(\lambda),
\end{equation}
where $H_k$ is as in \eqref{eq:definition-ld-rv}. 
By Lemma \ref{lem:liouville-mgf-asymptotics-bounds} we already know that the term on the left hand side in \eqref{eq:ld-convergence} is bounded from below and above uniformly in $k\geq 1$. 
To prove the claimed convergence, we first write
\begin{equation}
\E[e^{-d^k H_k}]
= \E [e^{-d^k H_k} \mathbf{1}_{H_k \leq \bC}] 
+ \E [e^{-d^k H_k} \mathbf{1}_{H_k > \bC}]
\end{equation}
and observe that it suffices to prove that
\begin{equation}
\label{eq:ld-convergence-compact-set}
-\frac{1}{d^k} \log \E [e^{-d^k H_k} \mathbf{1}_{H_k \leq \bC}] \to h(\lambda)\in (0,\infty)
\end{equation}
once $\bC$ is large enough.
Indeed, assuming that \eqref{eq:ld-convergence-compact-set} holds, we have
\begin{align}
\label{eq:ld-indicator-upper}
-\frac{1}{d^k} \log \E [e^{-d^k H_k} ]  &\leq -\frac{1}{d^k} \log \E [e^{-d^k H_k} \mathbf{1}_{H_k \leq \bC}] , \\
\label{eq:ld-indicator-lower}
-\frac{1}{d^k} \log \E [e^{-d^k H_k} ]
&\geq -\frac{1}{d^k} \log \big( \E [e^{-d^k H_k} \mathbf{1}_{H_k \leq \bC}] + e^{-d^k \bC} \big).
\end{align}
Now, from \eqref{eq:ld-indicator-upper} together with \eqref{eq:ld-convergence-compact-set} we can deduce
\begin{equation}
\limsup_{k\to \infty} -\frac{1}{d^k} \log \E [e^{-d^k H_k} ] \leq h(\lambda) ,
\end{equation}
while \eqref{eq:ld-indicator-lower} together with \eqref{eq:ld-convergence-compact-set} implies for $\delta >0$
\begin{align}
\liminf_{k\to \infty} -\frac{1}{d^k} \log \E [e^{-d^k H_k} ]
&\geq
\liminf_{k\to \infty} \Big(- \frac{1}{d^k} \log
\Big(
\E [e^{-d^k H_k} \mathbf{1}_{H_k \leq \bC}]\big(1 + e^{-d^k \bC} \big( \E [e^{-d^k H_k} \mathbf{1}_{H_k \leq \bC}]\big)^{-1}\big)
\Big) \Big)\nnb
& \geq
\liminf_{k\to \infty}\Big( - \frac{1}{d^k} \log \E [e^{-d^k H_k} \mathbf{1}_{H_k \leq \bC} ]
- \frac{1}{d^k}\log \big (1 + e^{-d^k \bC} e^{d^k h(\lambda) (1+\delta) }
\big)\Big) \nnb
&\geq 
\liminf_{k\to \infty} - \frac{1}{d^k} \log \big(
\E [e^{-d^k H_k} \mathbf{1}_{H_k \leq \bC}] 
-  \limsup_{k\to \infty} \frac{1}{d^k} e^{-d^k \bC} e^{d^k h(\lambda)(1+\delta) } \nnb
&\geq
\liminf_{k\to \infty } - \frac{1}{d^k} \log \E [e^{-d^k H_k} \mathbf{1}_{H_k \leq \bC} ]
= h(\lambda)
\end{align}
for $\bC$ chosen large enough such that $\bC>h(\lambda)(1+\delta)$.

It remains to prove \eqref{eq:ld-convergence-compact-set} for which we use a large deviations argument similar to the proof \cite[Theorem 4.3.1]{MR1619036}.
The difference to this reference is that the rate function $I \colon \R\to [0,\infty]$ defined in \eqref{eq:rate-function-liminf} may not have compact level sets, which requires us to modify the proof of the lower bound in \eqref{eq:ld-convergence-compact-set}.
We first observe that, by \cite[Theorem 4.1.18]{MR1619036} (and the remark following it),
the rate function $I$ can be equivalently expressed through
\begin{equation}
\label{eq:rate-function-liminf}
I(x) = \lim_{\epsilon \to 0} \liminf_{k\to \infty} - \frac{1}{d^k} \log \P(H_k \in B(x,\epsilon) )
\end{equation}
and
\begin{equation}
\label{eq:rate-function-limsup}
I(x) = \lim_{\epsilon \to 0} \limsup_{k\to \infty} -  \frac{1}{d^k} \log \P(H_k \in B(x, \epsilon) ).
\end{equation}

For the upper bound in \eqref{eq:ld-convergence-compact-set} we observe that we have for any $x \in (0,\bC)$ and $\epsilon>0$ such that $B(x,\epsilon) \subseteq [0,\bC]$
\begin{align}
-\frac{1}{d^k} \log \E [e^{-d^k H_k} \mathbf{1}_{H_k \leq \bC}]
&\leq -\frac{1}{d^k} \log \E[e^{-d^k H_k} \mathbf{1}_{H_k \in B(x,\epsilon)}] 
\leq
-\frac{1}{d^k} \log \E[e^{-d^k (x+ \epsilon) } \mathbf{1}_{H_k \in B(x,\epsilon)}] \nnb
& \leq 
 (x+ \epsilon) - \frac{1}{d^k} \log \P( H_k \in B(x,\epsilon) ).
\end{align}
Taking $k\to \infty$ we obtain
\begin{equation}
\limsup_{k\to \infty}  -\frac{1}{d^k} \log \E [e^{-d^k H_k} \mathbf{1}_{H_k \leq \bC}] 
\leq x + \limsup_{k\to \infty} - \frac{1}{d^k} \log \P(H_k \in B(x, \epsilon) ) + \epsilon.
\end{equation}
Sending $\epsilon\to 0$ and using \eqref{eq:rate-function-limsup}, it follows that 
\begin{equation}
\limsup_{k\to \infty}  - \frac{1}{d^k} \log \E [e^{-d^k H_k} \mathbf{1}_{H_k \leq \bC}] \leq x + I(x)
\end{equation}
Taking the infimum over $x\in [0,\bC)$ yields
\begin{equation}
\label{eq:ld-proof-upper}
\limsup_{k\to \infty}  -\frac{1}{d^k} \log \E [e^{-d^k H_k} \mathbf{1}_{H_k \leq \bC}] \leq  \inf_{x\in [0,\bC) } \{ x + I(x) \}=\inf_{x\in [0,\infty)} \{x+I(x)\},
\end{equation}
if $\bC$ is large enough, since $I(x)\geq 0$.

For the lower bound in \eqref{eq:ld-convergence-compact-set},
we fix $r<\infty$ and $\delta>0$ and denote by $\Psi_I(r) = \{ x\in \R \colon I(x) \leq r \}$ the level sets of $I$. 
Note that $I$ might not be a good rate function, i.e.\ the level sets are not necessarily compact.
We overcome this problem by considering the set $\Psi_I(r) \cap [0,\bC]$ instead.
This is a compact set, since $\Psi_I(r)$ is closed.
Since $I$ is lower-semicontinuous, there is a neighbourhood $A_x$ for any $x\in\R$, such that
\begin{equation}
\label{eq:ld-choice-A-x}
\inf_{y\in \bar A_x} I(y) \geq I(x) - \delta, \qquad A_x \subseteq B(x,\delta).
\end{equation}
Note that
\begin{equation}
\bigcup_{x\in \Psi_I(r)\cap [0,\bC] } A_x
\end{equation}
is an open cover of the compact set $\Psi_I(r) \cap [0,\bC]$. 
Thus, we can extract a finite subcover
\begin{equation}
\Psi_I(r) \cap [0,\bC] \subseteq \bigcup_{i=1}^N A_{x_i}
\end{equation}
for $x_i \in \Psi_I(r) \cap [0,\bC]$, $i=1,\ldots, N$.
Using \eqref{eq:ld-choice-A-x} we then have
\begin{align}
\label{eq:ld-lower-estimate-expectation}
\E [e^{-d^k H_k} \mathbf{1}_{H_k \leq \bC} ]
&= 
\E \big[ e^{-d^k H_k} \mathbf{1}_{ H_k \in [0,\bC] \cap (\cup_{i=1}^N A_{x_i} ) } \big]
+ 
\E \big[e^{-d^k H_k} \mathbf{1}_{ H_k \in [0, \bC] \cap \big( \cup_{i=1}^N A_{x_i}\big)^c } \big ] \nnb
& \leq 
\E [e^{-d^k H_k} \mathbf{1}_{H_k \in [0,\bC] \cap (\cup_{i=1}^N A_{x_i} )} ]
+ 
 \P \big( H_k \in [0, \bC] \cap \big( \cup_{i=1}^N A_{x_i} \big)^c  \big) \nnb
& \leq 
\sum_{i=1}^N \E \big[e^{-d^k H_k} \mathbf{1}_{H_k \in [0,\bC] \cap  A_{x_i} } \big]
+ 
\P \big( H_k \in [0, \bC] \cap \big( \cup_{i=1}^N A_{x_i}\big)^c  \big) \nnb
&\leq
\sum_{i=1}^N e^{-d^k (x_i-\delta)}  \P \big( H_k \in  \overline{ A_{x_i} } \big) 
+ 
 \P \big( H_k \in  [0,\bC] \cap \big( \cup_{i=1}^N A_{x_i}\big)^c  \big) 
\end{align}
Using the large deviations upper bound \eqref{eq:ld-rate-function-compact-sets} for the compact sets $\overline{A_{x_i}}$ and  $[0,\bC] \cap \big( \cup_{i=1}^N A_{x_i}\big)^c $, we obtain from \eqref{eq:ld-lower-estimate-expectation}
\begin{align}
\liminf_{k\to \infty } - \frac{1}{d^k} \log \E [e^{-d^k H_k} \mathbf{1}_{H_k \leq \bC} ]
&\geq \min \Big\{  \min_{i=1, \ldots, N} \{  x_i-\delta - \limsup_{k\to \infty}\frac{1}{d^k} \log \P(H_k \in \overline{ A_{x_i} } ) \} , \nnb
&  -\limsup_{k\to \infty } \frac{1}{d^k} \log \P(H_k \in [0,\bC] \cap \big( \cup_{i=1}^N A_{x_i}\big)^c)  \Big\} \nnb
& \geq
\min \Big\{   \min_{i=1,\ldots, N} \{ x_i-\delta + \inf_{y \in \overline{A_{x_i} } } I(y) \} , \, \inf_{y\in [0,\bC] \cap \big( \cup_{i=1}^N A_{x_i}\big)^c}  I(y) \Big\} \nnb
&\geq \min \Big\{   \min_{i=1,\ldots, N} \{ x_i-\delta + I(x_i) -\delta \} , \, r \Big\} .
\end{align}
In the last step we used \eqref{eq:ld-choice-A-x} and the observation
\begin{equation}
[0, \bC] \cap \Psi_I(r) \subseteq \bigcup_{i=1}^N A_{x_i}
\qquad \implies \qquad
\big( \bigcup_{i=1}^N A_{x_i} \big)^c\cap [0,\bC] \subseteq \Psi_I(r)^c,
\end{equation}
which implies that $I(x) >r$ for all $x \in \big( \cup_{i=1}^N A_{x_i} \big)^c \cap [0,\bC]$.
Taking $r\to \infty$ and then $\delta\to 0$ gives
\begin{equation}
\liminf_{k\to \infty } - \frac{1}{d^k} \log \E [e^{-d^k H_k} \mathbf{1}_{H_k \leq \bC} ]
\geq
\min_{i=1,\ldots, N} \{ x_i + I(x_i) \} \geq \inf_{x\in [0,\bC]} \{x + I(x) \}\geq \inf_{x\in [0,\infty)} \{x + I(x) \}.
\end{equation}
We thus see that \eqref{eq:ld-convergence-compact-set} holds with $h(\lambda)=\inf_{x\in [0,\infty)}\{x+I(x)\}$,
and the finiteness of $h(\lambda)$ is due to Lemma \ref{lem:liouville-mgf-asymptotics-bounds}.

It remains to prove that $h$ is continuous in $\lambda$ and that the convergence in \eqref{eq:liouville-mgf-asymptotics} is uniform on compact sets.
To this end, we prove that for $\lambda_1 \leq \lambda_2$ we have 
\begin{equation}
\label{eq:h-continuity}
h(\lambda_1) \leq h(\lambda_2) \leq \frac{\lambda_2}{\lambda_1} h(\lambda_1).
\end{equation}
This implies
\begin{equation}
0 \leq  h(\lambda_2) - h(\lambda_1)  \leq \frac{\lambda_2}{\lambda_1} h(\lambda_1) - h(\lambda_1) = \big( \frac{\lambda_2}{\lambda_1} - 1 \big) h(\lambda_1) ,
\end{equation}
from which the continuity readily follows.

To prove \eqref{eq:h-continuity} we first denote
\begin{equation}
h_k = - \frac{1}{d^k} \log \E[e^{-\lambda \base^k \MA } ].
\end{equation}
Since $h_k$ is increasing in $\lambda$ we have $h_k (\lambda_1) \leq h_k(\lambda_2)$ for all $k\geq 1$, so that the left inequality in \eqref{eq:h-continuity} follows when taking $k\to \infty$. 

To prove the right inequality, we apply Jensen's inequality to obtain
\begin{equation}
\E[e^{-\lambda_2 \base^k \MA} ] = \E[e^{-\frac{\lambda_2}{\lambda_1} \lambda_1 \base^k \MA} ] \geq \E[ e^{-\lambda_1 \base^k \MA }  ]^{\frac{\lambda_2}{\lambda_1}}.
\end{equation}
Thus, taking logarithms yields
\begin{equation}
h_k(\lambda_2) \leq \frac{\lambda_2}{\lambda_1} h_k(\lambda_1) ,
\end{equation}
and so the right inequality of \eqref{eq:h-continuity} follows again when taking $k\to \infty$.

Finally, we prove that the convergence \eqref{eq:liouville-mgf-asymptotics} is uniform on compact sets. By the previous lines, we have
\begin{equation}
h_k(\lambda_1) \leq h_k (\lambda_2) \leq \frac{\lambda_2}{\lambda_1} h_k(\lambda_1),
\end{equation}
which implies
\begin{equation}
0 \leq  h_k(\lambda_2) - h_k(\lambda_1)  \leq \frac{\lambda_2}{\lambda_1} h_k(\lambda_1) - h_k(\lambda_1)
= \big( \frac{\lambda_2}{\lambda_1} - 1 \big) h_k(\lambda_1) = \frac{(\lambda_2 - \lambda_1 )}{ \lambda_1} h_k(\lambda_1).
\end{equation}
Let now $K\subseteq (0, \infty)$ be a compact set.
Then
\begin{equation}
0 \leq h_k(\lambda_2) - h_k(\lambda_1) \leq C_K (\lambda_2 - \lambda_1), \qquad C_K = \sup_{\lambda \in K} \frac{1}{\lambda} h_k (\lambda),
\end{equation}
which shows that the family $(h_k\mid_K)_{k\in \N}$ is uniformly Lipschitz continuous with Lipschitz constant bounded by $C_K$.
It then follows that the convergence $h_k\to h$ as $k\to \infty$ is uniform on $K$.
\end{proof}

\subsection{Exponential moments of $\Mplus + \Mminus$:  Proof of Theorem \ref{thm:sinh-mgf-asymptotics}}
\label{ssec:sinh-mgf-asymptotics}

As we indicated earlier, our method of proof also applies to the negative exponential moments of $\Mplus + \Mminus$,
where we recall that $\Mplus$ and $\Mminus$ are constructed from the usual Branching Random Walks $S$ and $-S$.
The two key observation are that, first, the arguments underlying the subadditivity \eqref{eq:subadditivity-empirical-measure} do not require the Branching Random Walk to be balanced,
and second, Lemma \ref{lem:liouville-mgf-asymptotics-bounds} also holds with $\frac{1}{2} (\Mplus + \Mminus)$ in place of $\MA$, as we now show.
\begin{lemma}
\label{lem:sinh-mgf-asymptotics-bounds}
Let $\alpha$ be as in \eqref{eq:exponents-first-order}.
There exist constants $0< c \leq C$ and $t_0\geq 0$, such that for $t \geq t_0$
\begin{equation}
\label{eq:sinh-mgf-asymptotics-bounds}
e^{-C t^\alpha }\leq
\E[e^{- \frac{t}{2}( \Mplus + \Mminus ) } ]
\leq e^{- c t^\alpha }.
\end{equation}
\end{lemma}

\begin{proof}
The argument is similar to  the proof of Lemma \ref{lem:liouville-mgf-asymptotics-bounds}.
To see the upper bound in \eqref{eq:sinh-mgf-asymptotics-bounds}, it is enough to prove the upper bound in \eqref{eq:crude-small-ball-asymptotics} with $\Mplus+\Mminus$ replacing $\MA$. 
Using \eqref{eq:brw-gmc-recursion}, we fix $\ell \in \N$ be such that $s= r e^{-\frac{\gamma^2}{2}\ell}$ for some $r<1/2$, which is determined similarly to the proof of Lemma \ref{lem:liouville-mgf-asymptotics-bounds} (and in particular, 
is small enough, independent of $s$). Then,
\begin{equation}
\Mplus + \Mminus \stackrel{d}{=} \frac{e^{-\frac{\gamma^2}{2}\ell}}{d^\ell} \sum_{v\in \Gen{\ell}} (e^{\gamma S_v} \Mplus_v + e^{-\gamma S_v} \Mminus_v)
\geq \frac{e^{-\frac{\gamma^2}{2}\ell}}{d^\ell} \sum_{v\in \Gen{\ell}} ( \Mplus_v \mathbf{1}_{S_v\geq 0} + \Mminus_v \mathbf{1}_{S_v\leq 0}),
\end{equation}
where the pairs $(\Mplus_v, \Mminus_v)$ are independent and distributed as $(\Mplus,\Mminus)$, and independent of $(S_v)_{v\in \Gen{\ell}}$.
Since at least half of the $S_v$'s are either all positive or all negative, 
and the law of $\Mplus$ is identical to that of $\Mminus$,
we obtain for $\theta >0$
and some constant $c_\theta>0$
\begin{align}
\P(\Mplus + \Mminus \leq s) &\leq \P(\frac{1}{d^\ell} \sum_{i=1}^{d^\ell/2} M_{v_i}^S \leq  r) 
=  \P(\sum_{i=1}^{d^\ell/2} M_{v_i}^S \leq  r d^\ell ) 
\leq \E[ e^{- \theta \sum_{i=1}^{d^\ell/2}M_{v_i}^S } ]  e^{\theta d^\ell r}
\nnb
&\leq \E[e^{-\theta M^S}]^{d^\ell /2} e^{\theta r d^\ell}
\leq e^{-c_\theta' d^\ell/2} e^{\theta r d^\ell} 
\end{align}
where $(M_{v_i}^S)_{i=1, \ldots, d^\ell/2}$ are independent and distributed as $M^S$, and $c_\theta'>0$ since $\E M^S=1$.
Thus, the claimed upper bound follows after taking $r>0$ small enough.
For the lower bound, the argument is again similar to that used in the proof of the lower bound \eqref{eq:lower-bound-small-ball-good-event}, and we sketch the steps. Replace 
the event $\cA_{\ell,u}$ of \eqref{eq:good-event-lower-bound-small-ball} by

\begin{equation}
\label{eq:good-event-lower-bound-small-ball-bis}
\tilde \cA_{\ell,u} = \big\{ \forall i=0,\ldots, \ell-1 \colon |X_e| \leq \frac{u}{2^{\ell-i+1}} \text{ for all } e \in \Gen{i} \big\}.
\end{equation}
As in the proof of \eqref{eq:lower-bound-event-infty-norm}, we obtain that 
\begin{equation}
\label{eq:sinh-probability-good-event-lower-bound}
\P(\tilde \cA_{\ell, u})\geq e^{-\tilde C_u d^\ell}.
\end{equation}
for some constant $\tilde C_u>0$, which is independent of $\ell$.
Now, on the event $\tilde\cA_{\ell,u}$ we have that
\begin{equation}
\label{eq:sinh-good-event-lower-bound}
\Mplus + \Mminus \leq e^{\gamma u} \frac{1}{d^\ell} \sum_{v\in \Gen{\ell}} (\Mplus_v+\Mminus_v),
\end{equation}
where the pairs $(\Mplus_v,\Mminus_v)_{v\in \Gen{\ell}}$ are independent and identically distributed and independent of the event $\tilde \cA_{\ell, u}$.
Using that $\E[ \Mplus_v+\Mminus_v]=2$, and applying the law of large numbers,
we obtain the analogue of \eqref{eq:MLv-lower-bound-by-LLN} with $(\Mplus+\Mminus)/2$ replacing $M^A$.
Together with \eqref{eq:sinh-probability-good-event-lower-bound}, this completes the argument.
\end{proof}

Next, we cast again the exponential moment of $\frac{1}{2}(\Mplus + \Mminus)$ into an expectation over random variables
that satisfy a weak large deviations principle.
As before we let $t_k = \lambda \base^k$ for $\lambda >0$ and $\base$ as in \eqref{eq:base-sequence-tk}.
Then we have by \eqref{eq:brw-gmc-recursion}
\begin{equation}
\E[ e^{- \frac{\lambda \base^k}{2} (\Mplus + \Mminus) } ]
= \E[ e^{ -\frac{\lambda}{2} \sum_{v\in \Gen{k}} ( e^{\gamma S_v} \Mplus_v +  e^{-\gamma S_v} \Mminus_v  )}   ]
= \E[ e^{-d^k \tilde H_k } ] ,
\end{equation}
where now $(\Mplus_v, \Mminus_v)_{v\in \Gen{k}}$ is a collection of independent identically distributed random variables with $M_v^{\pm} \stackrel{d}{=} M^{\pm}$ and $\tilde H_k$ is defined by
\begin{equation}
\tilde H_k = -\frac{1}{d^\ldlv} \sum_{v\in \Gen{\ldlv}} \log \E[ e^{-\frac{\lambda}{2} ( e^{\gamma x } \Mplus +  e^{-\gamma x} \Mminus ) }  ]_{x=S_v} .
\end{equation}
Thus, the analogue to the real valued function $G$ from \eqref{eq:definition-G} is now $\tilde G \colon \R \to \R$ defined by
\begin{equation}
\label{eq:definition-G-sinh}
\tilde G (x) = -\log \E[e^{- \frac{\lambda}{2} (e^{\gamma x } \Mplus + e^{-\gamma x} \Mminus )} ].
\end{equation}

Then, analogously to Lemma \ref{lem:continuity-G}, we have the following estimates for the function $\tilde G$.  
\begin{lemma}
\label{lem:continuity-G-sinh}
Let $\tilde G$ be as in \eqref{eq:definition-G-sinh} and let $y\in B(x,\delta)$ for $x\in \R$ and $\delta>0$. Then
\begin{equation}
\label{eq:continuity-G-sinh}
\tilde G(x) e^{-\gamma \delta  } \leq \tilde G(y) \leq \tilde G(x)e^{\gamma \delta}
\end{equation}
\end{lemma}

\begin{proof}
Since $r\mapsto r^{e^{-\gamma \delta}}$ is concave, we have by Jensen's inequality and $e^{2\gamma \delta} \geq 1$ 
\begin{align}
\E[e^{- \frac{\lambda}{2} ( e^{\gamma y} \Mplus + e^{-\gamma y} \Mminus ) }]
&\leq
\E[e^{- \frac{\lambda}{2} ( e^{\gamma (x-\delta) } \Mplus + e^{-\gamma (x+\delta) } \Mminus )   }] \nnb
&\leq
\E[e^{- \frac{\lambda}{2} e^{-\gamma \delta} ( e^{\gamma x} \Mplus + e^{-\gamma x } \Mminus ) } ] 
 \leq \E[e^{-\frac{\lambda}{2} ( e^{\gamma x} \Mplus +  e^{-\gamma x } \Mminus ) }]^{e^{-\gamma \delta} }
\end{align}
and similarly
\begin{equation}
\E[ e^{-\frac{\lambda}{2} ( e^{\gamma y} \Mplus + e^{-\gamma y} \Mminus ) } ] 
\geq
\E[ e^{- \frac{\lambda}{2} ( e^{\gamma x} \Mplus + e^{-\gamma x} \Mminus ) } ]^{e^{\gamma \delta}}.
\end{equation}
Thus, taking the logarithm, the claim follows.
\end{proof}

\begin{proof}[Proof of Theorem \ref{thm:sinh-mgf-asymptotics}]
Following the same arguments as in the proof of Proposition \ref{prop:weak-ldp}, we have that the laws of $\tilde H_k$ satisfy a large deviations principle with rate function
\begin{equation}
\label{eq:rate-function-sinh}
\tilde I(x) = \sup_{A\in \cA \colon x \in A} - \liminf_{k\to \infty}  \frac{1}{d^k} \tilde f(k, A) ,
\end{equation}
where $\tilde f (k,A) = \log \P(\tilde H_k \in A)$.
Indeed, the arguments leading to the subadditivity only use the continuity of $G$, which holds for $\tilde G$ by Lemma \ref{lem:continuity-G-sinh}.
In particular, the fact that $A$ is balanced is not used in the proof.
Thus, the there is a function $\tilde h \colon (0,\infty) \to \R$ such that the convergence \eqref{eq:sinh-mgf-asymptotics} holds.
\end{proof}

\section{Decay of the $L^1$ norm and decay of correlations}

In this section we prove the main results for the measure $\PLv{t_k}$,
which are the decay of the $L^1$ norm in Theorem \ref{thm:pt-liouville-decay-l1} and the weak decay of correlations in Corollary \ref{cor:decay-correlation-tree}.
The idea is to use 
the expression \eqref{eq:pt-liouville-definition} and bound separately the numerator (with the event $B$ being that the $L^1$ norm of $(A_v)_{v\in D_{k-a}}$ is not small) and the denominator. The bound on the numerator uses the strict convexity of the functions
\begin{equation}
\label{eq:hk-definition}
h_k(\lambda) = -\frac{1}{d^k} \log\E[e^{-\lambda \base^k \MA}] ,
\end{equation}
where $\base$ is as in \eqref{eq:base-sequence-tk},
which unfortunately we prove only for specific $\lambda$.  From the (quantitative) strict convexity we then get a quantitative control of the numerator in \eqref{eq:pt-liouville-definition}. The details of the latter are in Section \ref{subsec-dec}.

Toward the convexity analysis alluded to above, we first prove the following variational property of the function $h$ from Theorem \ref{thm:liouville-mgf-asymptotics}.
\begin{lemma}
Let $h$ be as in Theorem \ref{thm:liouville-mgf-asymptotics} and let $\lambda>0$. Then
\begin{equation}
\label{eq:variational-formula-limit-mgf}
h(\lambda) = \min_{y \in \R^d: \sum_{i=1}^d  y_i=0} \frac{\sum_{i=1}^d h(\lambda e^{\gamma y_i} )  }{d}.
\end{equation}
\end{lemma}

\begin{proof}
Let $\bar h$ denote the right hand side of \eqref{eq:variational-formula-limit-mgf}.
Clearly, we have $\bar h (\lambda) \leq h(\lambda)$. 
Using \eqref{eq:balanced-brw-gmc-recursion} with $n=1$
and denoting the increment in the first level as $X=(X_i)_{i=1}^d$,
so that $\sum_{i=1}^d X_i=0$, we have
\begin{align}
h_k(\lambda) &=
 - \frac{1}{d^k} \log \E[e^{-\lambda p^{k-1}\sum_{i=1}^d e^{\gamma X_i} \hMA_i}] \nnb
&= - \frac{1}{d^k} \log \E_X \big[ \prod_{i=1}^d \E [ e^{ - \lambda p^{k-1} e^{\gamma X_i} \MA } \mid X_i ]   \big] \nnb
&=- \frac{1}{d^k} \log \E_X \big[ e^{-d^{k-1}\sum_{i=1}^d h_{k-1}(\lambda e^{\gamma X_i})}\big],
\end{align}
where we introduced $\tMA_i$, $i=1,\ldots,d$ identically distributed as $\MA$ and  independent of each other and of  $X$.
By Theorem \ref{thm:liouville-mgf-asymptotics} the convergence $h_k \to h$ is uniform on compact sets.
Thus, for every compact set $K \subseteq \R$, we have
\begin{align}
\label{eq:hk-upper-bound-compact}
h_k(\lambda) \leq - \frac{1}{d^k} \log \E_X \big[ {\bf 1}_{\{X_i\in K, i=1,\ldots,d\}} e^{-d^k\left(\frac{\sum_{i=1}^d h(\lambda e^{\gamma X_i})}{d}+o_k(1)\right)}\big],
\end{align}
where the term $o_k(1)$ depends on the set $K$ and vanishes as $k\to\infty$.
Let now $y^* = y^*(\lambda)\in \R^d$ with $\sum_{i=1}^d y^*_i=0$ be such that
\begin{equation}
\bar h(\lambda) = \frac{\sum_{i=1}^dh(\lambda e^{\gamma y^*_i})}{d}.
\end{equation}
Note that $y^*$ exists, since $h$ is continuous and $\sum_{i=1}^d h(\lambda e^{\gamma y_i}) \to \infty$ for $\|y\|_\infty\to \infty$ satisfying the constraint $\sum y_i = 0$.
Let $\epsilon>0$ and let $K \subseteq \R$ be a compact set such that $y_0 \in K^\circ$,
where $K^\circ$ denotes the interior of the set $K$.
Since $h$ is continuous,
there is $\delta>0$ such that if $|y-y^*| \leq \delta$ then
\begin{equation}
\Big| \frac{\sum_{i=1}^d h(\lambda e^{\gamma y^*_i} )}{d} - \frac{\sum_{i=1}^d h(\lambda e^{\gamma y_i} )}{d} \Big| \leq \epsilon.
\end{equation}
Then we have from \eqref{eq:hk-upper-bound-compact}, with $\delta > 0$ such that $\{y: \|y-y^*\|_\infty <\delta\}\subset K$
\begin{align}
h_k(\lambda) &\leq
- \frac{1}{d^k} \log \E_X \big[ {\bf 1}_{\{X_i\in K, i=1,\ldots,d\}} e^{-d^k\left(\frac{\sum_{i=1}^d h(\lambda e^{\gamma X_i})}{d}+o_k(1)\right)}\big] \nnb
&\leq 
- \frac{1}{d^k} \log \E_X \big[ {\bf 1}_{\{\|X-y^*\|_\infty<\delta\}}
 e^{-d^k\left(\frac{\sum_{i=1}^d h(\lambda e^{\gamma X_i})}{d}+o_k(1)\right)}\big] \nnb
&\leq \frac{\sum_{i=1}^d h(\lambda e^{\gamma y^*_i})}{d}+\epsilon-\frac{1}{d^k} 
\log \P \big( \|X-y^*\|_\infty<\delta)+o_k(1).
\end{align}
Sending $k\to \infty$, we arrive at
\begin{equation}
h(\lambda) \leq \bar h(\lambda ) + \epsilon=\min_{y\in \R^d: \sum_{i=1}^d y_i=0}\frac{\sum_{i=1}^d h(\lambda e^{\gamma y_i})}{d}+\epsilon .
\end{equation}
Since $\epsilon>0$ was arbitrary, the claimed equality \eqref{eq:variational-formula-limit-mgf} follows.
\end{proof}

\begin{remark}
The same proof works with $\lambda e^{\gamma x}$ for $x \in \R$ in place of $\lambda$. Thus, we also have 
\begin{equation}
\label{eq:h-exp-variational}
h(\lambda e^{\gamma x}) = \min_{y\in \R^d:\sum_{i=1}^d y_i=0} \frac{\sum_{i=1}^d h(\lambda e^{\gamma (x+y_i)})}{d}.
\end{equation}
\end{remark}

In what follows, we consider the function $f_\lambda \colon \R \to \R$ defined in \eqref{eq:f-tau-definition-2},
which we recall is given by
\begin{equation}
\label{eq:f-tau-definition-2a}
f_\lambda(x) = h(\lambda e^{\gamma x} ) ,
\end{equation}
where $h \colon \R \to \R$ is as in Theorem \ref{thm:liouville-mgf-asymptotics}.
We first show the function $f_\lambda$ is convex.
\begin{lemma}
\label{lem:f-variational-convex}
Let $\lambda > 0$ and let $f_\lambda$ be defined by \eqref{eq:f-tau-definition-2a}. Then $f_\lambda \colon \R \to \R$ is convex. 
\end{lemma}

\begin{proof}
We first check that
the function $f_\lambda$ is Jensen-convex (or midconvex, in the terminology of \cite{MR442824}), that is
$$f_\lambda(\frac{x+y}{2})\leq \frac{f_\lambda(x)+f_\lambda(y)}{2}.$$
Indeed, set $\bar x=(x+y)/2$, $\Delta=(x-y)/2$, and $\tilde \lambda=\lambda e^{\gamma \bar x}$. Then, 
\begin{align}
\frac{f_\lambda(x)+f_\lambda(y)}{2}=\frac{f_{\tilde \lambda}(\Delta)+f_{\tilde \lambda}(-\Delta)}{2}
\geq \min_{y\in \R^d: \sum_{i=1}^d y_i=0} \frac{\sum_{i=1}^d f_{\tilde \lambda}(y_i)}{d}=f_{\tilde \lambda}(0)=
f_\lambda(\bar x),
\end{align}
where the inequality uses \eqref{eq:h-exp-variational} with $\lambda$ replaced by $\tilde \lambda$.
Further, $f_\lambda$ is also continuous, and therefore easily shown to be convex; see for instance \cite[p.\ 220]{MR442824}.
\end{proof}

In the proof of Theorem \ref{thm:pt-liouville-decay-l1} we need to establish a lower bound on
\begin{equation}
\label{eq:sum-mgf-to-lower-bound}
\frac{1}{d^{k-a}} \sum_{v \in \Gen{k-a} } \Lambda( \lambda \base^a e^{\gamma A_v} ),
\end{equation}
where $a>0$ is an integer,
in terms of the same quantity with $A_v = 0$ for all $v\in \Gen{k-a}$.
It is tempting to use the convergence
\begin{equation}
\label{eq:f-tau-from-mgf}
\lim_{a \to \infty} \frac{1}{d^a} \Lambda (\lambda \base^a e^{\gamma x} ) = h(\lambda e^{\gamma x}) = f_\lambda(x)
\end{equation}
to replace $\Lambda(\lambda \base^a e^{\gamma A_v})$ by $d^a f_\lambda(A_v)$
and then use $\sum_v A_v = 0$ together with the strict convexity of $f_\lambda$ at $0$.
However, the limit \eqref{eq:f-tau-from-mgf} is not uniform in the value of $(A_v)_{v\in \Gen{k-a}}$
and thus, a more careful treatment is necessary.

In what follows we write, again with $a\in \N$,
\begin{equation}
g_a(A_v) = \Lambda( \lambda \base^a e^{\gamma A_v} ).
\end{equation}
By Lemma \ref{lem:liouville-mgf-asymptotics-bounds}, 
there are constants $\Clb>0$ and $\clb>0$ such that if $A_v \geq \Clb$ then
\begin{equation}
\label{eq:mgf-exp-lower-bound}
g_a(A_v) \geq d^a \clb e^{ \gamma \alpha A_v },
\end{equation}
where we used that $\base^{\alpha} = d$.

As already alluded earlier the convergence in \eqref{eq:f-tau-from-mgf} may not be uniform in $x\in \R$.
However, since the functions $\frac{1}{d^a} g_a$ are uniformly Lipschitz
on compacts by Lemma \ref{lem:continuity-G},
the convergence is uniform for all $x$ in a compact set.
Thus, for any $\delta > 0$ and any $\Ccomp>0$ there is $a_0=a_0(\delta, \Ccomp)$ large enough such that for all $a\geq a_0$
\begin{equation}
\big| \frac{1}{d^a} g_a - f_\lambda \big| \leq \delta
\end{equation}
uniformly on $[-\Ccomp, \Ccomp]$.

Let $\mu^{A,k-a}\equiv \muA$ be the empirical measure of $(A_v)_{v\in \Gen{k-a}}$,
i.e.\ $\muA$  is a random variable with values in $\cM_1(\R)$, the space of 
Borel probability measures on $\R$,
such that for any Borel function $\psi\colon \R \to \R$, 
\begin{equation}
\int_\R \psi d \muA = \frac{1}{d^{k-a}} \sum_{ v\in \Gen{k-a} } \psi(A_v).
\end{equation}
Note that with this notation we have
\begin{equation}
\lpnorm{A}{k-a}{1} = \int_\R |x| d \muA
\end{equation}
and moreover \eqref{eq:sum-mgf-to-lower-bound} reads
\begin{equation}
\frac{1}{d^{k-a}} \sum_{v\in \Gen{k-a}} \Lambda ( \lambda \base^a e^{\gamma A_v} ) = \int_\R g_a d \muA.
\end{equation}

We now state and prove the main technical result of this section.
\begin{proposition}
\label{prop:lower-bound-sum-mgf-strict-convexity}
Let $\epsilon > 0$ and assume
\begin{equation}
\label{eq:l1-norm-larger-epsilon}
\int |x| d\muA > \epsilon.
\end{equation}
Moreover, let $\lambda>0$ be such that $f_\lambda$ is strictly convex at $0$.
Then there is $a_0 \in \N$ large enough and $\delta > 0$ (both depending on $\epsilon$ only) such that for all $a\geq a_0$
\begin{equation}
\label{eq:lower-bound-sum-mgf-strict-convexity}
\frac{1}{d^a} \int_\R g_a d \muA \geq f_\lambda(0) + \delta .
\end{equation}
\end{proposition}
The proof of Proposition \ref{prop:lower-bound-sum-mgf-strict-convexity} is rather technical. The heart of the proof is based on the simple observation that if $\{A_v\}$ is a collection of random variables such that $\sum A_v=0$ and the average of $|A_v|$ is strictly positive
then the average of $e^{\theta A_v}$ for any $\theta \neq 0$ is strictly larger than $1$, and similarly the average of $g_a(A_v)=f(e^{\theta A_v})$ for $f$ increasing which is strictly convex at $1$, is larger than $f(1)$. We need however a quantitative version of this implication, which quantifies the gap in terms of the average of $|A_v|$. This would be easy if there existed a uniform upper bound on the $|A_v|$. In our setup, such an upper bound does not 
exist, and we have to be able to also control the gap in situations where most variables satisfy that $|A_v|$ is small and $A_v<0$, but there is a small number of $v$s for which $A_v$ is very large. The proof proceeds by introducing appropriate truncations (the events $\cD_C$ in \eqref{eq:mu-A-bounded-first-moment}, $\cE_K$ in \eqref{eq:event-exponential-right-tail-av-small}), and then we are able to control the bad event described above. We now give the details. 
\begin{proof}
We decompose $\muA$ into three parts according to
\begin{equation}
\label{eq:empirical-measure-decomposition}
\muA = \muA_{-, \Ccomp} + \muA_{0, \Ccomp} + \muA_{+, \Ccomp},
\end{equation}
where
\begin{align}
\supp(\muA_{-, \Ccomp}) \subseteq (-\infty, -\Ccomp), \qquad
\supp(\muA_{0, \Ccomp}) \subseteq [-\Ccomp, \Ccomp], \qquad
\supp(\muA_{+, \Ccomp}) \subseteq (\Ccomp, \infty).
\end{align}
Then the left hand side of \eqref{eq:lower-bound-sum-mgf-strict-convexity} decomposes into
\begin{equation}
\int_\R g_a \muA = \int_\R g_a d \muA_{-, \Ccomp} + \int_\R g_a d \muA_{0, \Ccomp} +   \int_\R g_a d \muA_{+, \Ccomp} .
\end{equation}
In what follows we use the notation
\begin{equation}
\mean_{\pm, \Ccomp} = \int_\R x d \muA_{\pm, \Ccomp} \qquad \text{and} \qquad \mean_{0, \Ccomp} = \int_\R x d \muA_{0, \Ccomp} 
\end{equation}
for the mean with respect to the measures $\muA_{\pm, K}$ and $\muA_{0,K}$. 
Since $\sum_{v \in \Gen{k-a}} A_v=0$, we have that
\begin{equation}
\label{eq:empirical-measure-sum-of-means}
\mean_{-, \Ccomp} + \mean_{0, \Ccomp} + \mean_{+, \Ccomp} = 0.
\end{equation}
Moreover, since $\muA$ is a probability measure, we have
\begin{equation}
\label{eq:empirical-measure-sum-of-masses}
|\muA_{-, \Ccomp}| + |\muA_{0, \Ccomp}|  + |\muA_{+, \Ccomp}|  = 1,
\end{equation}
where we use the notation $|\mu| = \int_\R d\mu$ for measures $\mu$ on $\R$.

We first argue that it is enough to prove \eqref{eq:lower-bound-sum-mgf-strict-convexity} on the event
\begin{equation}
\label{eq:mu-A-bounded-first-moment}
\cD_{\BigC}= \Big\{  \int |x| d\muA(x) \leq \BigC \Big\} ,
\end{equation}
where $\BigC>0$ is sufficiently large.
Indeed, since $\int x d\muA = 0$, we have on $\cD_\BigC^c$
\begin{equation}
\int_0^\infty x d \muA(x) \geq \BigC/2,
\end{equation}
and thus, we have by the lower bound \eqref{eq:mgf-exp-lower-bound} and Jensen's inequality
\begin{equation}
\frac{1}{d^a} \int g_a d\muA
\geq \frac{1}{d^a} \int_0^\infty g_a(x) d \muA(x)
\geq \clb \int_0^\infty e^{\alpha \gamma x} d\muA(x)
\geq \clb q e^{\frac{1}{q} \alpha \gamma  \BigC/2} ,
\end{equation}
where we denoted
\begin{equation}
\muA( [0, \infty) ) = q.
\end{equation}
Note that if
\begin{equation}
\label{eq:condition-for-bound-increasing-in-positive-mass}
\frac{\alpha \gamma \BigC}{2} >1
\end{equation}
the map $(0,1) \to \R$, $q \mapsto \clb q e^{\frac{1}{q} \alpha \gamma  \BigC/2}$ is decreasing.
Thus, for any $\BigC>0$ which satisfies \eqref{eq:condition-for-bound-increasing-in-positive-mass}, we have
\begin{equation}
\int g_a d \muA \geq \clb e^{\alpha \gamma \BigC/2}.
\end{equation}
Choosing $\BigC$ large enough, we can make sure that
\begin{equation}
\clb e^{\alpha \gamma \BigC/2} > 2 f_\lambda(0) ,
\end{equation}
so that for any $\epsilon>0$, \eqref{eq:lower-bound-sum-mgf-strict-convexity} holds with any $\delta \in (0, f_\lambda(0))$.
We emphasise that $\BigC$ depends neither on $a\in \N$ nor on $\epsilon>0$, but only on $\clb$.

Next, we argue that it suffices to prove \eqref{eq:lower-bound-sum-mgf-strict-convexity} on the event
\begin{equation}
\label{eq:event-exponential-right-tail-av-small}
\cE_{\Ccomp} = \{ |\muA_{+, \Ccomp}| e^{ \alpha \gamma \frac{\mean_{+, \Ccomp}}{ |\muA_{+, \Ccomp}| } } \leq \frac{2 f_\lambda(0)}{\clb} \}
\end{equation}
for $\Ccomp$ large enough.
Indeed, by the lower bound \eqref{eq:mgf-exp-lower-bound} and Jensen's inequality we have on $\cE_{\Ccomp}^c$ for any $\Ccomp \geq \Clb$, 
\begin{equation}
\int g_a d \muA_{+, \Ccomp}
\geq d^a \clb \int e^{\alpha \gamma x} d \muA_{+, \Ccomp}(x) 
\geq d^a | \muA_{+, \Ccomp} | \clb e^{ \alpha \gamma \frac{\mean_{+, \Ccomp}}{ |\muA_{+, \Ccomp}| } } \geq d^a 2 f_\lambda(0) .
\end{equation}
Thus, on $\cE_{\Ccomp}^c$ for $\Ccomp\geq \Clb$ we have that for any given $\epsilon >0$, \eqref{eq:lower-bound-sum-mgf-strict-convexity} holds again with any $\delta \in (0,f_\lambda(0) )$.
In summary, we showed that, no matter the choice of $\epsilon>0$, for $\Ccomp \geq \Clb$ and $\BigC$ chosen according to \eqref{eq:condition-for-bound-increasing-in-positive-mass}
we have on $\cD_{\BigC}^c \cup \cE_{\Ccomp}^c$ for any $a\in \N$ and any $\delta \in (0,f_\lambda(0) )$
\begin{equation}
\frac{1}{d^a}\int g_a d \muA \geq f_\lambda(0) + \delta.
\end{equation}

It remains to prove that \eqref{eq:lower-bound-sum-mgf-strict-convexity} holds on $\cD_{\BigC}\cap \cE_{\Ccomp}$.
Note that we can still choose $K=K(\epsilon)$ and $a_0=a_0(\epsilon)$ depending on $\epsilon$.
We first argue that on $\cE_{\Ccomp}$ the contribution coming from $\muA_{+, \Ccomp}$ is negligible provided $\Ccomp$ is large enough.
To this end, we observe that, since $\mean_{+, \Ccomp} \geq |\muA_{+, \Ccomp}| \Ccomp$,
we have from the condition in \eqref{eq:event-exponential-right-tail-av-small}
\begin{equation}
| \muA_{+, \Ccomp} | e^{\gamma \alpha \Ccomp} \leq \frac{ 2 f_\lambda(0) }{\clb}.
\end{equation}
So, as $\Ccomp \to \infty$, we have that
\begin{equation}
\label{eq:empirical-measure-large-positive-mass-to-0}
| \muA_{+, \Ccomp}| \to 0
\end{equation}
uniformly in the values of the configuration $(A_v)_{v\in \Gen{k-a}}$ and all $a\in \N$.

Then we get from the condition in \eqref{eq:event-exponential-right-tail-av-small} with $\caux =\frac{2f_\lambda (0)}{\clb}$
\begin{equation}
\alpha \gamma \frac{\mean_{+, \Ccomp}}{|\muA_{+, \Ccomp}|} + \log |\muA_{+\Ccomp}| \leq \log \caux ,
\end{equation}
so that 
\begin{equation}
\alpha \gamma \mean_{+, \Ccomp} \leq | \muA_{+, \Ccomp} | \log \caux -  | \muA_{+, \Ccomp} |  \log | \muA_{+, \Ccomp} |. 
\end{equation}
Hence, we also have  
\begin{equation}
\label{eq:empirical-measure-positive-mass-to-0}
\mean_{+, \Ccomp} \to 0
\end{equation}
as $k\to \infty$ uniformly in the values of the configuration $(A_v)_{v\in \Gen{k-a}}$ and $a\geq 1$.

We also record that on $ \cD_{\BigC}$ we have
\begin{equation}
\label{eq:empirical-measure-bulk-to-1}
|\muA_{0,\Ccomp}| \to 1
\end{equation}
as $\Ccomp \to \infty$ uniformly in the values of the configuration $(A_v)_{v\in \Gen{k-a}}$.
Indeed, Markov's inequality implies
\begin{equation}
\muA \big(  (-\Ccomp, \Ccomp)^c\big) \leq \frac{\int |x| d\muA(x)}{\Ccomp}
\leq \frac{\BigC}{\Ccomp}.
\end{equation}
from which \eqref{eq:empirical-measure-bulk-to-1} readily follows.

Fix now $\epsilon>0$. With the flexibility to choose $K=K(\epsilon)$ and $a_0=a_0(K,\epsilon)$, we need to find $\delta>0$ such that \eqref{eq:lower-bound-sum-mgf-strict-convexity} holds.
Since $f_\lambda$ is convex, we have that for some subdifferential $\subdif \in \R$ of $f_\lambda$ at $0$,
\begin{equation}
\label{eq:subdifferential-lower-bound-general}
f_\lambda (x) \geq f_\lambda(0) + \subdif x.
\end{equation}
Such a $\subdif$ exists by the convexity of $f_\lambda$, but it may not be unique. 

Since $f_\lambda$ is strictly convex at $0$, we have that for any $s>0$ there is $\Delta_s>0$ such that either
\begin{equation}
\label{eq:lower-bound-subdifferential-positive}
f_\lambda(x) \geq f_\lambda(0)  + \subdif x+ \Delta_s, \qquad x\geq s
\end{equation}
or
\begin{equation}
\label{eq:lower-bound-subdifferential-negative}
f_\lambda(x) \geq f_\lambda(0)  + \subdif x+ \Delta_s, \qquad x \leq - s.
\end{equation}
We take $s_\epsilon=\epsilon/8$ and choose $\delta>0$ such that
\begin{equation}
\label{eq:choice-delta}
\delta = \frac{\Delta_{s_\epsilon}}{4} .
\end{equation}
In addition, we choose $\deli = \deli(\epsilon) > 0$ such that
\begin{equation}
\subdif \deli \leq \delta \,\,\,(= \Delta_{s_\epsilon}/4) \qquad \text{and} \qquad \deli \leq \epsilon/8 .
\end{equation}
Using \eqref{eq:empirical-measure-positive-mass-to-0} we then choose $\Ccomp$ large enough, such that 
\begin{equation}
\label{eq:choice-of-compact}
\mean_{+, \Ccomp} \leq \deli \qquad \text{and} \qquad 
|\muA_{0,\Ccomp}| \geq 1- \frac{\delta}{f_\lambda(0)} .
\end{equation}
By \eqref{eq:empirical-measure-sum-of-means} we then have
\begin{equation}
\label{eq:empirical-measure-mean-without-large-positives-lower-bound}
\mean_{0,\Ccomp} + \mean_{-, \Ccomp} \geq -\deli ,
\end{equation}
which also implies
\begin{equation}
\label{eq:empirical-mean-compact-lower-bound}
\mean_{0,\Ccomp} \geq -\deli.
\end{equation}
Since $f_\lambda$ is increasing, we can take $\delii = \delii(\delta) > 0$ such that
\begin{equation}
\label{eq:f-tau-increasing-at-0}
f_\lambda(\delii) \geq  f_\lambda (0) + 3\delta.
\end{equation}
Now we split into the event
\begin{equation}
\cG_{\Ccomp}= \{ \mean_{-, \Ccomp} \leq - (\deli + \delta_2) \}
\end{equation}
and its complement. 
We first establish the desired bound on $\cG_{\Ccomp}^c$.
By \eqref{eq:empirical-measure-mean-without-large-positives-lower-bound} we have on this event
\begin{equation}
\label{eq:empirical-measure-mean-on-compact-positive}
\mean_{0, \Ccomp} \geq \delta_2 > 0.
\end{equation}
Since $f_\lambda$ is convex, we have by Jensen's inequality
\begin{align}
\label{eq:empirical-measure-convexity-f}
 \int f_\lambda d\muA_{0, \Ccomp}   
& \geq |\muA_{0, \Ccomp}| f_\lambda \big( \frac{\mean_{0, \Ccomp}}{|\muA_{0,\Ccomp}|} \big) .
\end{align}
Using the convexity of $f_\lambda$ again allows to estimate
\begin{equation}
|\muA_{0,\Ccomp}| f_\lambda \big( \frac{\mean_{0, \Ccomp}}{\muA_{0, \Ccomp}} \big) + f(0) (1 - |\muA_{0,\Ccomp}| )
\geq f_\lambda ( \mean_{0, \Ccomp} ),
\end{equation}
and thus, using \eqref{eq:empirical-measure-mean-on-compact-positive} and \eqref{eq:f-tau-increasing-at-0} and the right hand side of \eqref{eq:choice-of-compact}, we have
\begin{align}
\label{eq:lower-bound-mass-on-compact}
|\muA_{0,\Ccomp}| f_\lambda \big( \frac{\mean_{0, \Ccomp}}{\muA_{0, \Ccomp}} \big)
&\geq f_\lambda ( \mean_{0, \Ccomp} )
 - f_\lambda(0) (1 - | \muA_{0, \Ccomp} |)
 \geq f_\lambda(\delii) - f_\lambda(0) + | \muA_{0, \Ccomp} | f_\lambda (0) \nnb
 &\geq 3\delta  + f_\lambda(0) -\delta = f_\lambda(0) + 2 \delta.
\end{align}
Together with \eqref{eq:empirical-measure-convexity-f} this gives
on the event $\cG_{\Ccomp}^c$
\begin{equation}
\int f_\lambda d \muA_{0,\Ccomp} 
\geq f_\lambda(0) + 2\delta .
\end{equation}
Taking $a_0 = a_0(K,\delta)$ large enough, such that for $a \geq a_0$
\begin{equation}
\label{eq:approximation-ga-choice-a0}
\frac{1}{d^a} g_a \geq f - \delta
\end{equation}
uniformly on $(-\Ccomp, \Ccomp)$,
we then have
\begin{equation}
\frac{1}{d^a} \int g_a d \muA
\geq \frac{1}{d^a} \int g_a  d \muA_{0,\Ccomp} 
\geq
\int f_\lambda d\muA_{0,\Ccomp} -\delta \geq f_\lambda(0) + \delta.
\end{equation}
Thus, we showed that on the event $\cD_{\BigC}\cap \cE_{\Ccomp} \cap\cG_{\Ccomp}^c$ the inequality \eqref{eq:lower-bound-sum-mgf-strict-convexity} holds with $\delta$ as chosen in \eqref{eq:choice-delta}. 

It remains to establish \eqref{eq:lower-bound-sum-mgf-strict-convexity} on the event $ \cD_{\BigC}\cap \cE_{\Ccomp} \cap \cG_{\Ccomp}$, i.e.\ we can assume
\begin{equation}
\mean_{-, \Ccomp} \geq - (\deli + \delii).
\end{equation}
Then, since $\mean_{0, \Ccomp} + \mean_{-, \Ccomp} \leq 0$ by
\eqref{eq:empirical-measure-sum-of-means},
\begin{equation}
\mean_{0, \Ccomp} \leq \deli+ \delii.
\end{equation}
We now use the strict convexity of $f_\lambda$ at $0$.
Assume that \eqref{eq:lower-bound-subdifferential-positive} is the case.
The other case \eqref{eq:lower-bound-subdifferential-negative} can be treated similarly.
Let $s_\epsilon$ be as chosen above \eqref{eq:choice-delta}.
Then we have for $x\geq s_\epsilon$
\begin{equation}
f_\lambda(x) \frac{s_\epsilon}{x} + f_\lambda(0) (1-\frac{s_\epsilon}{x})
\geq f(s_\epsilon) \geq f_\lambda(0) + \subdif s_\epsilon +\Delta_{s_\epsilon} .
\end{equation}
Rearranging this yields
\begin{equation}
\label{eq:subdifferential-improved-lower-bound}
f_\lambda(x) \geq f_\lambda(0) + \subdif x + \Delta_{s_\epsilon} \frac{x}{s_\epsilon}.
\end{equation}
Thus, using \eqref{eq:subdifferential-lower-bound-general} and the improved lower bound \eqref{eq:subdifferential-improved-lower-bound}
\begin{align}
\int f_\lambda d\muA_{0,\Ccomp} &= \frac{1}{d^{k-a} } \sum_{v\colon -\Ccomp \leq A_v \leq \Ccomp} f_\lambda (A_v) 
=\frac{1}{d^{k-a} } \Big( \sum_{v\colon -\Ccomp \leq A_v \leq s_\epsilon} f_\lambda (A_v) 
+  \sum_{v\colon s_\epsilon \leq A_v \leq \Ccomp} f_\lambda (A_v)  \Big) \nnb
& \geq \frac{1}{d^{k-a}} \Big( \sum_{v\colon -\Ccomp \leq A_v \leq s_\epsilon} (f_\lambda (0) + \subdif A_v)
+ \sum_{v\colon s_\epsilon \leq A_v \leq \Ccomp}( f_\lambda(0) + \subdif A_v + \Delta_{s_\epsilon} \frac{A_v}{s_\epsilon} )
\Big) \nnb
& = | \muA_{0,\Ccomp} | f_\lambda (0) + \subdif \mean_{0,\Ccomp} + \frac{\Delta_{s_\epsilon}}{s_\epsilon} \frac{1}{d^{k-a} } \Big( \sum_{v\colon A_v \geq 0} A_v - \sum_{v\colon 0 < A_v \leq s_\epsilon} A_v - \sum_{v\colon A_v \geq \Ccomp} A_v \Big)
\end{align}
Using $\sum_{v\in \Gen{k-a}}A_v=0$ so that 
\begin{equation}
\frac{1}{d^{k-a}} \sum_{A_v \geq 0} A_v = \frac{1}{2} \lpnorm{A}{k-a}{1}
\end{equation}
and estimating the other two sums accordingly, we have 
\begin{align}
\int f_\lambda d\muA_{0,\Ccomp} &\geq | \muA_{0,\Ccomp} | f_\lambda (0) + \subdif \mean_{0,\Ccomp} 
+ \frac{\Delta_{s_\epsilon} }{s_\epsilon}
\Big(\frac{\lpnorm{A}{k-a}{1}}{2} - s_\epsilon - \mean_{+,\Ccomp} \Big). 
\end{align}
Then, using \eqref{eq:approximation-ga-choice-a0} and the  assumption \eqref{eq:l1-norm-larger-epsilon},
\begin{align}
\frac{1}{d^a} \int g_a d \muA
&\geq \frac{1}{d^a} \int g_a d \muA_{0, \Ccomp} \geq \int f_\lambda d \muA_{0,\Ccomp} - \delta \nnb
&\geq | \muA_{0,\Ccomp} | f_\lambda (0) + \subdif \mean_{0,\Ccomp} 
+ \frac{\Delta_{s_\epsilon}}{s_\epsilon} \big( \epsilon/2 - \epsilon/8 - \mean_{+,\Ccomp} \big) - \delta.
\end{align}

By the choices \eqref{eq:choice-of-compact}, \eqref{eq:choice-delta} and the lower bound \eqref{eq:empirical-mean-compact-lower-bound} we then have
\begin{align}
\frac{1}{d^a} \int g_a d \muA
& \geq f_\lambda (0) - \delta -  \delta + \frac{\Delta_{s_\epsilon}}{s_\epsilon} ( \epsilon/4 - \deli  )  - \delta \nnb
& \geq f_\lambda (0) + 4\frac{ \Delta_{s_\epsilon} }{\epsilon} (\epsilon/2 - \epsilon/8 -  \epsilon/8) - 3\delta \nnb
&= f_\lambda(0) + \Delta_{s_\epsilon} - 3 \delta
\geq f_\lambda(0) + 4 \delta - 3 \delta = f_\lambda(0)  + \delta
\end{align}
as needed.
\end{proof}

\subsection{Decay of the $L^1$ norm: proof of Theorem \ref{thm:pt-liouville-decay-l1} }

We collected all auxiliary results to give the proof of Theorem \ref{thm:pt-liouville-decay-l1}.
The key ingredient is Proposition \ref{prop:lower-bound-sum-mgf-strict-convexity},
which allows to bound the numerator of \eqref{eq:pt-liouville-definition} by a term which goes to $0$ faster than the denominator.

\begin{proof}[Proof of Theorem \ref{thm:pt-liouville-decay-l1}]
We show that for a given $\epsilon>0$, there exists $a>0$ large enough such
that as $k \to 0$
\begin{equation}
\label{eq:pt-probability-to-0}
\PLv{t_k} ( \lpnorm{A}{k-a}{1} > \epsilon ) \to 0.
\end{equation}
From the definition of $\PLv{t_k}$ in \eqref{eq:pt-liouville-definition} we see that we need an upper bound for
\begin{equation}
\label{eq:pt-l1-large-numerator}
\E [ \mathbf{1}_{ \{ \lpnorm{A}{k-a}{1} > \epsilon \}}  e^{-t_k \MA }] ,
\end{equation}
which, as $k\to \infty$, goes to $0$ faster than $\E[e^{-t_k \MA }]$.

To this end, we expand $\MA$ in \eqref{eq:pt-l1-large-numerator} at level $k-a$ according to \eqref{eq:balanced-brw-gmc-recursion}, which yields
\begin{equation}
\label{eq:decomp-for-decay-L1}
\MA \stackrel{d}{=} \frac{1}{d^{k-a}}  \sum_{v\in \Gen{k-a}} e^{\gamma A_v -\frac{\gamma^2}{2} (k -a)} \MA_v,
\end{equation}
where $(\MA_v )_{v\in \Gen{k -a}}$ is a collection of independent identically distributed random variables with $\MA_v \stackrel{d}{=} M^A$.
Then we have by the decomposition \eqref{eq:decomp-for-decay-L1},
the independence of $(A_v)_{v\in \Gen{k-a}}$ and $(\MA_v)_{v\in \Gen{k -a}}$, and the tower property of conditional expectation,
\begin{equation}
\E \big[ \mathbf{1}_{ \{ \lpnorm{A}{k-a}{1} > \epsilon \}}  e^{-t_k \MA } \mid (A_v)_{v\in \Gen{k-a}} \big]
 = \mathbf{1}_{ \{ \lpnorm{A}{k-a}{1} \geq \epsilon \} } e^{-\sum_{v\in \Gen{k-a} } \Lambda(\lambda p^a e^{\gamma A_v} ) } .
\end{equation}
By Proposition \ref{prop:lower-bound-sum-mgf-strict-convexity}, on the event
$\{ \lpnorm{A}{k-a}{1} > \epsilon \}$
there exist $\delta = \delta(\epsilon)>0$ and $a_0=a_0(\epsilon)$ such that
for $a\geq a_0$,
\begin{equation}
\label{eq:lower-with-gain-bound-sum-mgfs}
\frac{1}{d^{k-a}}\sum_{v\in \Gen{k-a} } \Lambda(\lambda p^a e^{\gamma A_v} ) \geq d^a (f_\lambda (0) + \delta).
\end{equation}
Using this we obtain that for $k$ large and $a>a_0$,
\begin{align}
\E \big[ \mathbf{1}_{ \{ \lpnorm{A}{k-a}{1} > \epsilon \}}  e^{-t_k \MA } \mid (A_v)_{v\in \Gen{k-a}} \big]
&\leq \mathbf{1}_{ \{ \lpnorm{A}{k-a}{1} \geq \epsilon \} } e^{- d^k (f_\lambda(0) + \delta) }
\end{align}
so that
\begin{equation}
\label{eq:pt-bound-numerator}
\E \big[ \mathbf{1}_{ \{ \lpnorm{A}{k-a}{1} > \epsilon \}}  e^{-t_k \MA } \big] \leq e^{-d^k (f_\lambda(0) + \delta )}  .
\end{equation}
Noting that $f_\lambda (0) = h(\lambda)$ we have by Theorem \ref{thm:liouville-mgf-asymptotics}  as $k\to \infty$
\begin{equation}
\E[e^{-t_k\MA}] = e^{-d^k (f_\lambda (0) + o(1) ) } ,
\end{equation}
where $o(1) \to 0$ as $k\to \infty$.
Thus, from \eqref{eq:pt-liouville-definition} and \eqref{eq:pt-bound-numerator} we obtain for $k$ large enough such that $|o(1)| \leq \delta/2$ and $a>a_0$,
\begin{equation}
\PLv{t_k}(\lpnorm{A}{k-a}{1} > \epsilon) \leq e^{-d^k(\delta/2) } ,
\end{equation}
which converges to $0$ as $k\to \infty$ as desired.
\end{proof}

\subsection{Decay of correlations: proof of Corollary \ref{cor:decay-correlation-tree}}
\label{subsec-dec}

\begin{proof}[Proof of Corollary \ref{cor:decay-correlation-tree}]
By the definition of $\embdPLv{t_k}$ and \eqref{eq:pt-liouville-definition} we need to show that for a suitable family of events $(\evnt)_{v\in \cL,\,a ,k\in \N }$ we have
\begin{equation}
\label{eq:liouville-correlation-eps-limit}
\lim_{a \to \infty }\limsup_{k\to \infty} \sup_{w\in \cL \colon d_\TT(v,w)=2(a+1)} 
\Big| \frac{\E[ \embdbalBRW_v \embdbalBRW_w \mathbf{1}_{\evnt} e^{-\lambda \base^k \MA}]}{\E[e^{-\lambda \base^k \MA} ]} \Big| =0.
\end{equation}

Fix now $w\in \cL$ with $d_\TT(v,w)=2(a+1)$, which means that the common ancestor of $v$ and $w$ is $a+1$ levels above $v$ and $w$.
Recall that we denote by $\anc{v}{a}$ and $\anc{w}{a}$ for $a\leq k$ the ancestors of $v$ and $w$ in $\rGen{a}{k}$,
where $\rGen{a}{k}$ denotes the set of vertices in $\cR_k \subseteq \TT$ at level $-a$ as defined below \eqref{eq:levels-of-infinite-tree}.
Note also that $\anc{w}{a}$ can take only one of $d-1$ values (given $v$), and that $d_\TT(\anc{v}{a},\anc{w}{a})=2$.
We therefore write
\begin{align}
\label{eq:decay-correlations-increment-x}
\embdbalBRW_{v} &= \embdbalBRW_{\anc{v}{a}} + \incr{v}{a}, \\
\label{eq:decay-correlations-increment-y}
\embdbalBRW_{w} & = \embdbalBRW_{\anc{w}{a}} + \incr{w}{a},
\end{align}
where $\incr{v}{a}$ and $\incr{w}{a}$ are the displacements along the paths $\anc{v}{a}\leftrightarrow v$ and $\anc{w}{a} \leftrightarrow w$.
Let now
\begin{equation}
\label{eq:correlation-good-event}
\evnt = \big\{ |\embdbalBRW_{ \anc{v}{a}}|\leq d\sqrt{\epsilon_a}, |\embdbalBRW_{z}| \leq d \sqrt{\epsilon_a }\; \mbox{\rm for all possible $z\in \rGen{a}{k}$ with $d_\TT(z,\anc{v}{a})=2$} \big\},
\end{equation}
where $\epsilon_a$ is as in Theorem \ref{thm:pt-liouville-decay-l1} and we note that there are precisely $d-1$ possible choices of $z$ in the right side of \eqref{eq:correlation-good-event}.

We first prove that the probabilities of these events converge to $1$ as $k\to \infty$ and then $a\to\infty$.
To this end, we note that by the exchangeability of $(\embdbalBRW_u)_{u\in \rGen{a}{}}$, we have
\begin{equation}
\embdELv{t_k} \big[ |\embdbalBRW_u|\mathbf{1}_{ \{\rlpnorm{\embdbalBRW}{a}{k}{1} \leq \epsilon_a\} } \big]
= \embdELv{t_k} \big[ |\embdbalBRW_{u'} |\mathbf{1}_{ \{\rlpnorm{\embdbalBRW}{a}{k}{1} \leq \epsilon_a \} } \big]
\end{equation}
for all $u, u' \in \rGen{a}{k}$, and thus
\begin{align}
\embdELv{t_k}[ |\embdbalBRW_u|\mathbf{1}_{ \{\rlpnorm{\embdbalBRW}{a}{k}{1} \leq \epsilon_a\} } ]
&= \frac{1}{d^{k-a}} \sum_{u\in \rGen{a}{k}} \embdELv{t_k}[ |\embdbalBRW_u|\mathbf{1}_{\{\rlpnorm{A}{a}{k}{1} \leq \epsilon_a\}  } ] \nnb
&= \embdELv{t_k}[ \rlpnorm{\embdbalBRW}{a}{k}{1} \mathbf{1}_{ \{\rlpnorm{\embdbalBRW}{a}{k}{1} \leq \epsilon_a\} } ] \leq \epsilon_a .
\end{align}
Then we have by Theorem \ref{thm:pt-liouville-decay-l1} and Markov's inequality 
\begin{align}
\label{eq:decay-correlations-event-complement-small}
\embdPLv{t_k}( \evnt^c )
&= \embdPLv{t_k}( \evnt^c ,\, \rlpnorm{\embdbalBRW}{a}{k}{1} \leq \epsilon_a ) + \embdPLv{t_k}( \evnt^c , \, \rlpnorm{\embdbalBRW}{a}{k}{1} > \epsilon_a )  \nnb
& \leq d \cdot \embdPLv{t_k} ( |\embdbalBRW_v| > d\sqrt{\epsilon_a}, \, \rlpnorm{\embdbalBRW}{a}{k}{1} \leq \epsilon_a  )
+ \embdPLv{t_k} ( \rlpnorm{\embdbalBRW}{a}{k}{1}> \epsilon_a  )  \nnb
&\leq d\cdot \frac{ \embdELv{t_k}[ |\embdbalBRW_{v}|\mathbf{1}_{ \{ \rlpnorm{\embdbalBRW}{a}{k}{1} \leq \epsilon_a \} } }{ d\sqrt{\epsilon_a} }+ o_k(1)  \leq \sqrt{\epsilon_a} + o_k(1).
\end{align}
It  follows that
\begin{equation}
\liminf_{a\to\infty} \liminf_{k\to \infty} \PLv{t_k}( \evnt) = 1
\end{equation}
as claimed.

We now turn to the proof of \eqref{eq:liouville-correlation-eps-limit}.
With the notation introduced in (\ref{eq:decay-correlations-increment-x}--\ref{eq:decay-correlations-increment-y}) we have 
\begin{align}
\label{eq:correlation-expand-level-k-a}
\E[\embdbalBRW_{v} \embdbalBRW_{w} \mathbf{1}_{\evnt} e^{- \lambda \base^k \MA} ] 
&= \E[ ( \embdbalBRW_{\anc{v}{a}} + \incr{v}{a} ) ((\embdbalBRW_{\anc{w}{a}} + \incr{w}{a} )  \mathbf{1}_{\evnt} e^{-\lambda \base^k \MA } ] \nnb
&= \E[\embdbalBRW_{\anc{v}{a}} \embdbalBRW_{\anc{w}{a}} \mathbf{1}_{\evnt} e^{- \lambda \base^k \MA }] \nnb
&\qquad+ \E\big[\big(\embdbalBRW_{\anc{v}{a}} \incr{w}{a} + \embdbalBRW_{\anc{w}{a}} \incr{v}{a} \big) \mathbf{1}_{ \evnt} e^{-\lambda \base^k \MA } \big] \nnb
&\qquad+ \E[ \incr{v}{a} \incr{w}{a} \mathbf{1}_{\evnt}  e^{-\lambda \base^k \MA}].
\end{align}
We first argue that the expectation values in \eqref{eq:correlation-expand-level-k-a} containing $\incr{v}{a}$ and $\incr{w}{a}$ vanish.
To this end, we use \eqref{eq:balanced-brw-gmc-recursion} and decompose the multiplicative chaos $\MA$ according to
\begin{equation}
\MA \stackrel{d}{=} \frac{1}{d^{k -a}} \sum_{u \in\rGen{a}{k} } e^{\gamma \embdbalBRW_u - \frac{\gamma^2}{2} (k-a)}\MA_u .
\end{equation}
For the second term in \eqref{eq:correlation-expand-level-k-a}, we condition on $(\embdbalBRW_v)_{v\in \rGen{a}{k}}$ and obtain by the independence of $\incr{w}{a}$ and $(\MA_v)_{v\in \rGen{a}{k} \setminus \{ \anc{w}{a} \}}$
\begin{align}
\label{eq:liouville-correlation-one-increment}
&\E[\embdbalBRW_{\anc{v}{a}} \incr{w}{a} \mathbf{1}_{\evnt} e^{-\lambda \base^k \MA}] = 
\E \big[ \E[ \embdbalBRW_{\anc{v}{a}}  \incr{w}{a} \mathbf{1}_{ \evnt} e^{- \lambda \base^k \MA } \mid (\embdbalBRW_v)_{v\in \rGen{a}{k} } ] \big] \nnb
&=\E \Big [ \embdbalBRW_{\anc{v}{a}} \mathbf{1}_{ \evnt}  \E \big[ e^{-\lambda \base^k \frac{1}{d^{k-a}} \sum_{u \neq \anc{w}{a}} e^{\gamma \embdbalBRW_u - \frac{\gamma^2}{2} (k -a)} \MA_u} \bigm | (\embdbalBRW_u)_{u\in \rGen{a}{k}} \big] \nnb
&\qquad \times   
\E[\incr{w}{a}  e^{-\lambda \base^k \frac{1}{d^{k-a}} e^{\gamma z -\frac{\gamma^2}{2}(k-a)}\MA_{\anc{w}{a}} } ]_{z=\embdbalBRW_{\anc{w}{a}}} \Big ].
\end{align}
Let $\cR_{a}^{\anc{w}{a}}  \subseteq \cR_k$ be the subtree of depth $a$ rooted at $\anc{w}{a}$ and let $\rGen{a}{k}^{\anc{w}{a}}$ be the set of descendants of $\anc{w}{a}$ at level $0$. 
Since $(\embdbalBRW_u^{\anc{w}{a}})_{ u\in \cR_{a}^{\anc{w}{a}} }$ is a balanced Branching Random Walk of depth $a$ started from $\anc{w}{a}$, we have
\begin{equation}
\label{eq:sum-exchangeable-terms}
\sum_{ u\in \rGen{a}{k}^{\anc{w}{a}} } \E[\embdbalBRW_u^{\anc{w}{a}} e^{-\lambda \base^k \frac{1}{d^{k-a}} e^{\gamma z -\frac{\gamma^2}{2}(k-a)}\MA_{\anc{w}{a}} } ] = 0, \qquad z\in \R.
\end{equation}
By the exchangeability of $(\embdbalBRW_u^{\anc{w}{a}})_{ u\in \rGen{a}{k}^{\anc{w}{a}} }$, all terms in the sum in \eqref{eq:sum-exchangeable-terms} have the same value, i.e.\ for all $u,u' \in \rGen{a}{k}^{\anc{w}{a}}$, it holds
\begin{equation}
\E[\embdbalBRW_u^{\anc{w}{a}} e^{-\lambda \base^k \frac{1}{d^{k-a}} e^{\gamma z -\frac{\gamma^2}{2}(k-a)}\MA_{\anc{w}{a}} } ] 
=\E[\embdbalBRW_{u'}^{\anc{w}{a}} e^{-\lambda \base^k \frac{1}{d^{k-a}} e^{\gamma z -\frac{\gamma^2}{2}(k-a)}\MA_{\anc{w}{a}} } ] .
\end{equation}
Therefore, since $\incr{w}{a}= \embdbalBRW_{u_0}^{\anc{w}{a}}$ for some $u_0\in {\rGen{a}{k}^{\anc{w}{a}} }$, we have
\begin{equation}
\label{eq:expectation-increment-zero}
\E[\incr{w}{a}  e^{- \lambda \base^k \frac{1}{d^{k-a}} e^{\gamma z -\frac{\gamma^2}{2}(k-a)}\MA_{\anc{w}{a}} } ]=0,
\end{equation}
from which it follows that the display on the right hand side of \eqref{eq:liouville-correlation-one-increment} is $0$ as claimed.

For the last expectation value in \eqref{eq:correlation-expand-level-k-a} the same idea together with the independence of $\MA_{\anc{v}{a} }$ and $\MA_{\anc{w}{a}}$ leads to
\begin{align}
\label{eq:expectation-two-increments}
\E[ \incr{v}{a} \incr{w}{a} &\mathbf{1}_{\evnt}  e^{-\lambda \base^k \MA }]
\nnb
&=\E \Big[
\mathbf{1}_{ \evnt}  \E \big[ e^{-\lambda \base^k \frac{1}{d^{k-a}} \sum_{u \neq \anc{w}{a}, \anc{v}{a} } e^{\gamma \embdbalBRW_u - \frac{\gamma^2}{2} (k -a)} \MA_u} \bigm | (\embdbalBRW_u)_{u\in \rGen{a}{k}} \big] 
\nnb
 &\qquad \times\E[ \incr{v}{a}  e^{-\lambda \base^k \frac{1}{d^{k-a}} e^{\gamma z -\frac{\gamma^2}{2}(k-a)}\MA_{\anc{v}{a}} } ]_{ z= \embdbalBRW_{\anc{v}{a}}} \nnb
&\qquad \times \E[ \incr{w}{a}  e^{- \lambda \base^k \frac{1}{d^{k-a}} e^{\gamma w -\frac{\gamma^2}{2}(k-a)}\MA_{\anc{w}{a}} } ]_{w=\embdbalBRW_{\anc{w}{a}}} 
\Big].
\end{align}
Now the inner expectations in \eqref{eq:expectation-two-increments} vanish again by \eqref{eq:expectation-increment-zero}.
Dividing now \eqref{eq:correlation-expand-level-k-a} by $\E[e^{-\lambda \base^k\MA}]$ 
we thus have
\begin{align}
\label{eq:correlation-level-k-a}
\big| \embdELv{t_k}[ \embdbalBRW_{v} \embdbalBRW_{w} \mathbf{1}_{\evnt}  ] \big|
=\big| \embdELv{t_k}[ \embdbalBRW_{\anc{v}{a}} \embdbalBRW_{\anc{w}{a}} \mathbf{1}_{\evnt}  ] \big|
\leq \epsilon_a .
\end{align}
In particular, the right hand side is uniform in $k\in\N$ and hence,
\begin{equation}
\limsup_{k\to \infty}\big| \embdELv{t_k}[ \embdbalBRW_{v} \embdbalBRW_{w} \mathbf{1}_{\evnt}  ] \big| \leq \epsilon_a .
\end{equation}
Since $\epsilon_a \to 0$ as $a\to \infty$, the limit \eqref{eq:decay-correlation-tree} follows.
\end{proof}

\section*{Acknowledgements}
We thank Roland Bauerschmidt, R\'emi Rhodes, Antoine Tilloy and Vincent Vargas for useful discussions and motivation.
We thank Nikolay Barashkov and Joona Oikarinen for useful exchanges and for discussing with us their work \cite{Barashkov2024SmallDeviations},
and Oren Louidor for making \cite{Fels2024BRW-1-2} available to us.
We thank the referees for a careful and thorough reading of the first version of this article, and for their detailed remarks.
This work was supported by Israel Science Foundation grants number 421/20 and 615/24.

MH greatfully acknowledges the additional financial support by the German Research Foundation (DFG).

\bibliography{decay_ssbrw.bbl}
\end{document}